\documentclass{amsart}

\usepackage{graphicx} 
\usepackage{mathptmx}
\usepackage{amssymb}
\usepackage{tikz}
\usetikzlibrary{arrows,automata}
\usepackage{tikz-cd}
\usepackage{hyperref}
\usepackage{placeins}
\usepackage{comment}

\usepackage{bm}
\usepackage{thmtools}

\hypersetup{
	colorlinks=true,
	linktocpage=true,
	linkcolor=[RGB]{161,1,49},
	citecolor=[RGB]{10,136,169},      
	urlcolor=cyan,
}

\newtheorem{theorem}{Theorem}[section]

\newtheorem{corollary}[theorem]{Corollary}
\newtheorem{lemma}[theorem]{Lemma}
\newtheorem{proposition}[theorem]{Proposition}

\newtheorem{definition}[theorem]{Definition}

\theoremstyle{definition}
\newtheorem{remark}[theorem]{Remark}

\newtheorem{example}[theorem]{Example}

\title[Extensibility and denseness of periodic monoid actions]{Extensibility and denseness of periodic monoid actions}

\author{Raimundo Briceño}
\address{Facultad de Matem\'aticas, Pontificia Universidad Cat\'olica de Chile. Santiago, Chile}
\email{raimundo.briceno@uc.cl}

\author{Álvaro Bustos-Gajardo}
\address{Departamento de Matem\'aticas, Fac. de Ciencias, Universidad de Chile. Santiago, Chile}
\email{alvaro.bustos@uchile.cl}

\author{Miguel Donoso-Echenique}
\address{University of M{\"u}nster, Einsteinstra{\ss}e 62, M{\"u}nster 48149, Germany}
\email{mdonosoe@uni-muenster.de,migueldonosoe@gmail.com}


 \date{}

\newcommand{\R}{\mathbb{R}}
\newcommand{\N}{\mathbb{N}}
\newcommand{\Z}{\mathbb{Z}}
\newcommand{\F}{\mathbb{F}}

\newcommand{\alphabet}{\mathcal{A}}

\newcommand{\symdiff}{\mathbin{\triangle}}
\newcommand{\acts}{\curvearrowright}
\newcommand{\Ext}{\operatorname{Ext}}

\newcommand{\im}{\operatorname{im}}
\newcommand{\supp}{\operatorname{supp}}
\newcommand{\Per}{P}
\newcommand{\Perg}{\partial_eP}

\newcommand{\sgmorph}{\theta}
\newcommand{\grmorph}{\theta}
\newcommand{\G}{\mathbf{G}}
\newcommand{\Periodic}{\operatorname{Per}}

\newcommand{\receivingSgroup}{\mathbf{G}}
\newcommand{\freeSgroup}{\bm{\Gamma}}
\newcommand{\freegroup}{\mathbf{F}}

\newcommand{\GrRightFrac}[1]{\freeSgroup_{#1}}

\newcommand{\genextmap}{\tau}

\newcommand{\GextensionSet}{\hat{X}}

\begin{document}

\begin{abstract}

 A monoid has the periodic approximation (PA) property if, for every full shift on a finite alphabet, finitely supported invariant probability measures are weak-* dense in the space of invariant Borel probability measures. We introduce the class of measure-theoretically extension-compatible (MEC) monoids, consisting of embeddable monoids for which every probability measure-preserving (p.m.p.) action admits an invertible p.m.p. extension. For an embeddable monoid $S$ with free $S$-group $(\Gamma,\gamma)$, we prove that the PA property implies MEC, and that the converse holds whenever $\Gamma$ has the PA property. We establish that every left reversible bicancellative monoid and every finitely generated free monoid is MEC, and deduce the PA property for left amenable monoids that are residually a finite group and for finitely generated free monoids. Moreover, for $S=\F_d^+$, we characterize the free $S$-group as the unique receiving $S$-group for which all fully supported invariant Markov measures on full shifts on finite alphabets admit invariant invertible extensions.
\end{abstract}


\maketitle
\setcounter{tocdepth}{2}
\tableofcontents

\section*{Introduction}

The study of periodic orbits is a fundamental aspect of the qualitative theory of dynamical systems. In particular, the question of whether periodic points are dense in a given system is relevant to topological and smooth dynamics \cite{artin1965periodic,devaney1986introduction}. The measure-theoretical analogue of this problem is whether periodic measures are dense in the set of invariant measures, a topic that has been extensively addressed in classical ergodic theory \cite{ruelle1973,sigmund1974}.

In the context of group actions, a concrete formulation of the first question is whether, for a countable group $G$ acting continuously on a compact metric space $X$, the set of periodic points $\Periodic(X,G)$ is dense in $X$. In this setting, the concept of periodicity must be adapted to the group framework, and the potential denseness of periodic points imposes algebraic constraints on $G$, specifically requiring it to be \emph{residually finite}. Moreover, denseness of periodic points characterizes residual finiteness in the following sense: a countable group $G$ is residually finite if and only if $\Periodic(X,G)$ is dense in $\alphabet^G$ for the shift action \cite[Theorem~2.7.1]{cellular_automata}.

More recently, interest has grown \cite{kechris2020weak,kechris2012weakcontainment,ren2018,shriver2023free} in the measure-theoretical analogue of this characterization: determining for which countable groups $G$ the set of periodic measures (that is, finitely supported invariant probability measures) is weak-* dense in the space of $G$-invariant Borel probability measures $M_G(\alphabet^G)$ for every finite alphabet $\alphabet$. Such groups are said to have the \emph{PA property}, and they are necessarily residually finite. However, not every residually finite group has the PA property \cite{kechris2012weakcontainment}. 

We may further ask if this still holds true for the subset of ergodic periodic measures; the groups with this property will be said to have the \emph{EPA property}. Clearly, the EPA property implies the PA property, but it is not known if the converse implication holds in general. It is known that all amenable, residually finite groups have the EPA property (see \cite{ren2018}, where denseness is established for more general systems with specification). On the other hand, Bowen showed that free groups on finitely many generators have the PA property \cite{bowen2003periodicity}.

These properties are intrinsic to the group, as the action is prescribed once the group is set. This is not the only case where the structure of the space of invariant measures tells us something about properties of the acting group (e.g., see \cite{glasner1997}).

In this work, we aim to extend these questions and results to the more general setting of monoid actions. For monoid actions, we call a point \emph{pre-periodic} if its orbit is finite, and \emph{periodic} if, in addition, every element of the monoid acts bijectively on that orbit. Our first result is the following extension of a characterization due to Ceccherini-Silberstein and Coornaert \cite[Proposition 2.14]{CeccheriniCoornaert2015} from the pre-periodic to the periodic setting.

\begin{restatable}{prop}{firstproposition}\label{thm:A}
		Let $S$ be a countable monoid. Then, $S$ is residually a finite group (resp. residually a finite monoid) if and only if the set of periodic (resp. pre-periodic) points is dense in $\alphabet^S$ for every finite alphabet $\alphabet$.
\end{restatable}

Next, we turn our attention to the corresponding measure-theoretical approximation problem, extending both the PA property and EPA property to monoids. To study these notions, we use invertible measure-preserving extensions. Roughly speaking, an embeddable monoid $S$ is \emph{measure-theoretically extension-compatible (MEC)} if every probability measure-preserving (p.m.p.) action of $S$ admits an invertible p.m.p. extension (see \S\ref{sec: MEC} for the formal definitions). When speaking of invertible extensions, we will always work in the context of an embedding $\eta$ of $S$ into a group $G$ such that $\langle\eta(S)\rangle=G$. In this situation, the pair $(G,\eta)$ is called a \emph{receiving $S$-group}, and referred to as the \emph{free $S$-group} if it is an initial object in the category of receiving $S$-groups. Our first main theorem establishes a connection between the PA/EPA and the MEC properties.

\begin{restatable}{thm}{theoremB}
\label{thm:characterization of the PA property}
Let $S$ be an embeddable monoid. If $S$ has the PA property, then $S$ is MEC. Moreover, if $\freeSgroup=(\Gamma, \gamma)$ is the free $S$-group and $\Gamma$ has the PA (resp. EPA) property, then the following are equivalent:
\begin{itemize}
\item[(i)] $S$ has the PA (resp. EPA) property.
\item[(ii)] $S$ is MEC.
\end{itemize}
\end{restatable}

We establish MEC for left reversible bicancellative monoids
and finitely generated free monoids. Applying Theorem~\ref{thm:characterization of the PA property} together with the corresponding approximation results for groups, we obtain the following.

\begin{restatable}{thm}{thirdtheorem}
\label{thm:C}
Let $S$ be an infinite monoid. The following holds:
\begin{enumerate}
\item If $S$ is left amenable and residually a finite group, then $S$ has the EPA property. 
\item If $S = \F_d^+$ for some positive integer $d$, then $S$ has the PA property.
\end{enumerate}
\end{restatable}

Finally, we provide a characterization of the free $S$-group for the free monoid $S=\F_d^+$ in terms of the existence of invertible extensions of its p.m.p. actions. In what follows, given a receiving $S$-group $\receivingSgroup=(G,\eta)$, we denote by $\Ext_{\receivingSgroup}(\alphabet^S,S)$ the set of all $S$-invariant measures $\mu$ on $\alphabet^S$ for which there exists an invertible extension, that is, a measure $\bar{\mu}\in M_G(\alphabet^G)$ satisfying $(\eta^*)_*\bar{\mu}=\mu$, where
$\eta^*\colon\alphabet^G\to\alphabet^S$ is the pullback map given by $\eta^*(x)=x\circ\eta$.


\begin{restatable}{thm}{fifththeorem}
\label{thm:E}
Let $S=\F^+_d$ be the free monoid  on $d \geq 1$ generators, let $\freegroup_d = (\F_d,\iota)$ be its free $S$-group, where $\iota\colon S\to\F_d$ is the canonical inclusion, and let $\receivingSgroup$ be a receiving $S$-group. The following are equivalent:
\begin{enumerate}
\item[(i)] $\receivingSgroup$ is a realization of the free $S$-group, i.e., $\receivingSgroup\simeq \freegroup_d$.
\item[(ii)] For every finite alphabet $\mathcal{A}$, we have $\Ext_{\receivingSgroup}(\alphabet^S,S)=M_S(\alphabet^S)$.
\item[(iii)] For every finite alphabet $\mathcal{A}$, every fully supported Markov measure in $M_S(\alphabet^S)$ belongs to $\Ext_{\receivingSgroup}(\alphabet^S,S)$.
\end{enumerate}
\end{restatable}

For $d \ge 2$, examples of non-free groups $G$ that admit embeddings of $S=\F_d^+$ with generating image include $\mathrm{BS}(1,2)$ and the free products $\F_{d-1}\ast \Z/n\Z$, for $n>1$. For every such receiving $S$-group $(G,\eta)$, Theorem \ref{thm:E} provides a finite alphabet $\alphabet$
and a fully supported $S$-invariant Markov measure on $\alphabet^S$ that admits no $G$-invariant extension to $\alphabet^G$ under the pullback map $\eta^*$. In contrast, the full shift $\alphabet^S$ admits the full $G$-shift $\alphabet^G$ as a topological $(G,\eta)$-extension.

The paper is organized as follows. In \S \ref{sec1}, we introduce the basic notions about embedding of monoids into groups and monoid actions. In \S \ref{sec2}, we discuss periodicity in the topological and measure-theoretical settings, and its connection to residual finiteness, leading to the proof of Proposition~\ref{thm:A}. In \S \ref{sec3}, we study invertible extensions and establish the connection with the PA property in Theorem~\ref{thm:characterization of the PA property}. In \S \ref{sec4}, we prove that left reversible bicancellative monoids and finitely generated free monoids are MEC, and deduce the EPA property for left amenable monoids that are residually a finite group and the PA property for finitely generated free monoids in Theorem~\ref{thm:C}. The section concludes with the characterization in Theorem~\ref{thm:E}.

\section{Preliminaries}
\label{sec1}

\subsection{Monoids and embeddings into groups}

A \textbf{monoid} is a set $S$ together with an associative binary operation $(s,t) \mapsto st$, and a---necessarily unique---identity element $1_S$. Throughout this work, we shall only deal with countable, discrete monoids.

A \textbf{submonoid} $T$ of $S$, denoted $T\leq S$, is a subset $T\subseteq S$ containing $1_S$ such that $tt'\in T$ for all $t,t'\in T$. We define the \textbf{submonoid of} $S$ \textbf{generated by} $K\subseteq S$ as follows:
$$\langle K\rangle^+=\{k_0\cdots k_{n-1} : n \in \N,~ k_0,\dots,k_{n-1}\in K\},$$
where the empty product is interpreted as $1_S$. We say that $K$ \textbf{generates} a submonoid $T\le S$ if $T=\langle K\rangle^+$.

A \textbf{monoid homomorphism} is a function $\sgmorph\colon S\rightarrow T$ satisfying $\sgmorph(st)=\sgmorph(s)\sgmorph(t)$ for every $s,t\in S$, and $\sgmorph(1_S)=1_T$. An \textbf{embedding} will be an injective homomorphism, and an \textbf{isomorphism} will be a bijective homomorphism. If there is an isomorphism $S\to T$, we say $S$ and $T$ are \textbf{isomorphic} and denote this by $S\simeq T$. In this case, $S$ and $T$ will be, in most cases, essentially the same for our purposes.

\subsubsection{Free monoids and presentations}

We want to give a monoid structure to the quotient of a monoid by an equivalence relation. An equivalence relation $\mathcal{R}$ on a monoid $S$ is said to be a \textbf{congruence} if for every $a,b,c\in S$, $a\mathbin{\mathcal{R}} b$ implies both $ca \mathbin{\mathcal{R}} cb$ and $ac \mathbin{\mathcal{R}} bc$. Given such a relation, we denote by $\pi_\mathcal{R}\colon S\rightarrow S/\mathcal{R}$ the quotient map $\pi_\mathcal{R}(s)=[s]_\mathcal{R}$.

If $\mathcal{R}$ is a congruence on $S$ and $a\mathbin{\mathcal{R}} a'$, $b\mathbin{\mathcal{R}} b'$,  transitivity implies that $ab \mathbin{\mathcal{R}} a'b'$, and thus the binary operation $S/\mathcal{R}\times S/\mathcal{R}\longrightarrow S/\mathcal{R}$ given by $[a]_\mathcal{R} \cdot [b]_\mathcal{R} =[ab]_\mathcal{R}$ is well-defined and gives $S/\mathcal{R}$ a monoid structure, for which $\pi_{\mathcal{R}}$ is a monoid homomorphism. For groups, the notions of quotient by a congruence and quotient by a normal subgroup are equivalent; this is easily seen by noting that there is a bijective correspondence associating to a congruence $\mathcal{R}$ the normal subgroup $N_{\mathcal{R}}=\{g\in G:g\mathrel{\mathcal{R}} 1_G\}\trianglelefteq G$ and vice versa. 

Given any set $B$, the \textbf{free monoid generated by} $B$ is the set $\F(B)^+$ of all finite words of elements from $B$, with concatenation as its binary operation, and the empty word $\varepsilon$ as the identity element. Given $R \subseteq\F(B)^+\times \F(B)^+$, the monoid $\langle B\mid R\rangle^+$ corresponds to the quotient $\F(B)^+/\mathcal{R}$, where $\mathcal{R}$ is the smallest congruence in $\F(B)^+$ containing $R$ as a subset. If $S=\langle B\rangle^+$ for some $B \subseteq S$, then $S$ is isomorphic to $\langle B\mid R\rangle^+$ for some $R\subseteq \mathbb{F}(B)^+\times \mathbb{F}(B)^+$. Note that if $\lvert B_1\rvert = \lvert B_2\rvert$, then $\mathbb{F}(B_1)^+ \simeq \mathbb{F}(B_2)^+$. Thus, we may speak of \emph{the} free monoid generated by $d$ elements, $\F_d^+:=\F(\{a_1,\dotsc,a_d\})^+$ for an arbitrary, fixed set $\{a_1,\dotsc,a_d\}$.

\subsubsection{Embeddable monoids}

Let $S$ be a monoid. A pair $\receivingSgroup = (G,\eta)$ is an $S$\textbf{-group} if $G$ is a group and $\eta\colon S\rightarrow G$ is a monoid morphism with $\langle \eta(S)\rangle =G$. If $\eta$ is an embedding, $\receivingSgroup$ is called a \textbf{receiving} $S$\textbf{-group}, and whenever a monoid admits an embedding into a group, we say $S$ is \textbf{embeddable}. A \textbf{morphism} between two $S$-groups $(G,\eta)$ and $(G',\eta')$ is a group morphism $\grmorph\colon G\rightarrow G'$ such that $\grmorph\circ \eta=\eta'$.

\begin{definition}
    Let $S$ be a monoid. A \textbf{free group on the monoid} $S$, or a \textbf{free} $S$\textbf{-group}, is an $S$-group $\freeSgroup = (\Gamma,\gamma)$ such that for every $S$-group $(G,\eta)$ there is a unique group morphism $\grmorph \colon \Gamma\rightarrow G$ with $\grmorph\circ \gamma=\eta$.
$$
\begin{tikzcd}
S \arrow[r, "\gamma"] \arrow[rd, "\eta"'] &  \Gamma \arrow[d, "\grmorph", dashed] \\
&  G                             
\end{tikzcd}
$$
\end{definition}

Note that, if $(\Gamma,\gamma)$ is the free $S$-group and $(G,\eta)$ is any $S$-group, the morphism $\grmorph\colon \Gamma\to G$ granted by the universal property of $(\Gamma,\gamma)$ must be surjective, as its image contains a generating set for $G$.

The free $S$-group always exists and it is unique up to isomorphism of $S$-groups (see \cite[\S 12]{CliffordII} for more details). Accordingly, we will speak of \emph{the} free $S$-group whenever we talk about properties that are stable under isomorphisms of $S$-groups, and of a \emph{realization} of the free $S$-group when we want to refer to a specific pair $(\Gamma,\gamma)$. 

If $S=\langle B\mid R\rangle^+$, the group $\langle B\mid R\rangle$ provides a concrete way of viewing the free $S$-group. More precisely, if $\mathcal{R}^+$ is the congruence generated by $R$ in $\F(B)^+$ and $\mathcal{R}$ is the congruence generated by $R$ in $\F(B)$, let
$$S=\F(B)^+/\mathcal{R}^+\quad \text{and}\quad \Gamma=\F(B)/\mathcal{R}.$$
Then, the pair $(\langle B\mid R\rangle,\iota)$, where $\iota\colon S\rightarrow \Gamma$ is the monoid morphism given by $\iota([b]_{\mathcal{R}^+})=[b]_{\mathcal{R}}$ for $b \in B$, is a realization of the free $S$-group. In this case, $(\langle B\mid R\rangle, \iota)$ will be referred to as the \textbf{canonical realization} of the free $S$-group (see \cite[Proposition~3.2]{bricenobustosdonoso1}). The following are examples of canonical realizations of free $S$-groups:
\begin{enumerate}
    \item $\Z^d=\langle a_1,\dots a_d\mid a_ia_j=a_ja_i\rangle$ for $\N^d=\langle a_1,\dots a_d\mid a_ia_j=a_ja_i \rangle^+$;
    \item $\text{BS}(m,n)=\langle a,b\mid ab^m=b^na\rangle$ for $\text{BS}(m,n)^+=\langle a,b\mid ab^m=b^na\rangle^+$;
    \item $\F_d=\langle a_1,\dotsc,a_d\mid\varnothing\rangle$ for $\F_d^+=\langle a_1,\dotsc,a_d\mid\varnothing\rangle^+$.
\end{enumerate}

A monoid $S$ can be embedded into a group if and only if it can be embedded into its free $S$-group \cite[Theorem~12.4]{CliffordII}. In the three instances above, $\iota$ is an embedding and the corresponding monoid is embeddable. 

Embeddable monoids are necessarily \textbf{bicancellative}, meaning that $a=b$ whenever $ac=bc$ or $ca=cb$ hold. None\-the\-less, not every bicancellative monoid is embeddable \cite{malcev}. A useful result due to Adian \cite{adjan1966defining} states that bicancellative one-relator finitely presented monoids are embeddable.

\subsubsection{The left reversible case}\label{rem:unique_receiving_group}

A monoid $S$ is said to be \textbf{left reversible} if $sS\cap tS\neq \varnothing$ for all $s,t\in S$. Groups and abelian monoids are left reversible. In contrast, the free monoid $\F_d^+$ is not left reversible for $d \geq 2$, since $a_1\F_d^+\cap a_d\F_d^+=\varnothing$.

In \cite[Theorem~1]{ore1931}, Ore showed that left reversibility is a sufficient condition for a bicancellative monoid to be embeddable. To describe this scenario, we need the following definition:

\begin{definition}
Let $S$ be a monoid. A receiving $S$-group $\receivingSgroup= (G,\eta)$ is a \textbf{group of right fractions} of $S$ if for every $g\in G$ there exist $s,t\in S$ such that $g=\eta(s)\eta(t)^{-1}$.
\end{definition}

In \cite{dubreil1943problemes}, Dubreil proved that if $S$ is a bicancellative monoid, then $S$ is left reversible if and only if a group of right fractions of $S$ exists. Whenever this happens, groups of right fractions of $S$ are the only receiving $S$-group up to isomorphism of $S$-groups, and they are hence all isomorphic to the free $S$-group \cite[Corollary 2.20 and Remark 3.4]{bricenobustosdonoso1}. As a consequence, the group of right fractions may thus be denoted by $\GrRightFrac{S} = (\Gamma_S,\gamma)$.

\begin{remark}
If $S$ is embeddable but not left reversible, then receiving $S$-groups are not necessarily unique modulo isomorphism of $S$-groups. This is shown in \cite[Example 2.27]{bricenobustosdonoso1}, where both $\text{BS}(1,2)$ and $\F_2$---paired with adequate embeddings---play the role of receiving $\F_2^+$-groups, but are non isomorphic even as groups.
\end{remark}

\subsection{Monoid actions}

We will understand an \textbf{action} $\alpha$ of $S$ on a set $X$, denoted as $S\overset{\alpha}{\acts} X$, by a function $\alpha\colon S\times X\rightarrow X$ such that $\alpha(s,\alpha(t,x))=\alpha(st,x)$ and $\alpha(1_S,x)=x$, for all $x \in X$ and $s,t\in S$. Most of the time $\alpha(s,x)$ will be written simply as $s \cdot x$, and the function $\alpha(s,\cdot)\colon X\rightarrow X$ as simply $\alpha_s$ or $s$. We also write $S \acts X$ if $\alpha$ is understood. An action $S\acts X$ is said to be \textbf{surjective} if $\alpha_s$ is surjective for each $s\in S$. Given $x \in X$, we define its $\textbf{$S$-orbit}$ as the $S$-invariant set $Sx = \{s \cdot x: s \in S\}$. A subset $A\subseteq X$ is $S$\textbf{-invariant} for an action $S\overset{\alpha}{\acts}X$ if $sA:=\{s\cdot x:x\in A\}\subseteq A$ for all $s\in S$, and \textbf{completely} $S$\textbf{-invariant} if $sA=A$ for every $s\in S$. In this case it makes sense to consider the action $S\acts A$ given by restricting $\alpha|_{S\times A}\colon S\times A\rightarrow A$. 

Given two monoids $S$ and $T$, a monoid morphism $\theta\colon S\to T$ and two actions $S\acts X$ and $T\acts Y$, a function $\varphi\colon Y\rightarrow X$ will be called $\theta$\textbf{-equivariant} if $\varphi(\theta(s)\cdot y)=s\cdot \varphi(y)$ for all $y\in Y$ and $s\in S$. If $S=T$ and $\theta=\text{id}_{S}$, we say that $\varphi$ is $S$\textbf{-equivariant}. 

\subsubsection{Structure preserving actions}

Let $X$ be a Polish space and $(X,\mathcal{B}(X))$ the associated standard Borel space. The space of Borel probability measures on $(X,\mathcal{B}(X))$ will be denoted by $M(X)$. We endow $M(X)$ with the weak-* topology, where convergence is characterized by: 
$$\mu_n\to \mu\iff \int_Xf\,d\mu_n\to \int_{X}f\,d\mu \quad \text{for all }f\in C_b(X).$$
If $X$ is compact, then $M(X)$ is compact. The \textbf{support} of $\mu \in M(X)$ is defined as
$$\supp(\mu)=\{x\in X:\mu(U)>0\text{ for every open neighborhood }U\text{ of }x\}.$$
The set $\supp(\mu)$ is closed and non-empty, and $\mu(A)=0$ whenever $A\cap \supp(\mu)=\varnothing$.

An action $S\overset{\alpha}{\acts} X$ is \textbf{Borel} if $\alpha_s$ is a Borel map for each $s\in S$. Given a Borel action $S\acts X$, we say that $\mu\in M(X)$ is $S$\textbf{-invariant} if $\mu(s^{-1}A)=\mu(A)$ for all $A\in \mathcal{B}(X)$ and $s\in S$. A \textbf{probability measure-preserving (p.m.p.) action} $S\acts(X,\mu)$ is a Borel action $S\acts X$ such that $\mu\in M(X)$ is $S$-invariant.

A Borel action $S\overset{\alpha}{\acts}X$ is a \textbf{continuous action} if $\alpha_s$ is continuous for each $s\in S$. Given a continuous action $S\acts X$, the space of $S$-invariant measures on $(X,\mathcal{B}(X))$, denoted by $M_S(X)$, is convex and weak-* closed in $M(X)$. Depending on $S\acts X$, the space $M_S(X)$ may be empty or not. A measure $\mu\in M_S(X)$ is \textbf{ergodic} if every set $A\in \mathcal{B}(X)$ such that $\mu(s^{-1}A\symdiff A)=0$ for all $s\in S$ satisfies $\mu(A)\in \{0,1\}$. The set of ergodic measures in $M_S(X)$ will be denoted by $\partial_e M_S(X)$.

\begin{lemma}
\label{lemma:surjective_support}
    If $S\acts X$ is a continuous action and $\mu\in M_S(X)$, then $\supp(\mu)$ is closed and $S$-invariant. If, moreover, $S\acts X$ is surjective and $\supp(\mu)$ is compact, then the restricted action $S\acts \supp(\mu)$ is surjective.
\end{lemma}

\begin{proof}
    Let us show that $\supp(\mu)$ is an $S$-invariant subset. Take $x \in \supp(\mu)$ and $s \in S$. Given an open neighborhood $U$ of $s \cdot x$, we have that $s^{-1}U$ is an open neighborhood of $x$, so $\mu(U) = \mu(s^{-1}U) > 0$. Thus $s\cdot x\in \supp(\mu)$. 

    To see that the action $S\acts \supp(\mu)$ is surjective, let $x\in \supp(\mu)$ and $s\in S$. For every $n\geq 1$ we have
    \[0<\mu(B(x,1/n))=\mu(s^{-1}B(x,1/n))=\mu(X\cap s^{-1}B(x,1/n)),\]
    so there exists $y_n\in \supp(\mu)\cap s^{-1}B(x,1/n)$. Taking a convergent subsequence $y_{n_k}\to y\in \supp(\mu)$ we get by the continuity of the action $S\acts X$ that $s\cdot y=\lim_k s\cdot y_{n_k}=x$.
\end{proof}

Consider Borel actions $S\overset{\alpha}{\acts} X$, $T\overset{\beta}{\acts} Y$, measures $\mu\in M(X)$, $\nu\in M(Y)$, a Borel map $\varphi\colon Y \rightarrow X$, and a monoid morphism $\theta\colon S\to T$. We say $\varphi$ is \textbf{measure-preserving} if $\varphi_*\nu=\mu$, where $\varphi_*\nu(A)=\nu(\varphi^{-1}A)$ for all $A\in \mathcal{B}(X)$. We also say that $\varphi$ is $\theta$\textbf{-equivariant} if $\varphi(\theta(s)\cdot y)=s\cdot \varphi(y)$ for all $y\in Y$ and $s\in S$. If $\mu$ is $S$-invariant and $\nu$ is $T$-invariant, we say that $S\overset{\alpha}{\acts} (X,\mu)$ is a $\theta$\textbf{-factor} of $T \overset{\beta}{\acts} (Y,\nu)$ (and that $\beta$ is a $\theta$\textbf{-extension} of $\alpha$) if there exist a $\mu$-conull $S$-invariant set $X_0\subseteq X$, a $\nu$-conull $T$-invariant set $Y_0 \subseteq Y$, and a measure-preserving $\theta$-equivariant map $\varphi\colon Y_0 \rightarrow X_0$. In this case, $\varphi$ is called a $\theta$\textbf{-factor map}. If $\varphi$ is moreover a bi-measurable map, we say that $\alpha$ and $\beta$ are $\theta$\textbf{-conjugate} and that $\varphi$ is a $\theta$\textbf{-conjugacy}. If $S\overset{\alpha}{\acts} X$ and $T\overset{\beta}{\acts} Y$ are continuous, we say that $\alpha$ is a \textbf{topological} $\theta$\textbf{-factor} (and that $\beta$ is a topological $\theta$\textbf{-extension} of $\alpha$) if there exists a surjective continuous $\theta$-equivariant map $\varphi\colon Y\to X$. In this case, $\varphi$ is called a \textbf{topological} $\theta$\textbf{-factor map}. If $\varphi$ is moreover a homeomorphism, we say that $\alpha$ and $\beta$ are \textbf{topologically} $\theta$\textbf{-conjugate} and that $\varphi$ is a \textbf{topological} $\theta$\textbf{-conjugacy}. If $S=T$ and $\theta=\text{id}_{S}$, we will drop the prefix ``$\theta$-'' in all of the above definitions.

\subsubsection{Shift actions}\label{sec:shift_actions}

Given a Polish space $K$, we consider the product space $K^S$ of all functions $x\colon S \to K$ endowed with the product topology. The continuous action $S \overset{\sigma}{\acts} K^S$ given by
$$
(s\cdot x)(t)=x(ts) \quad  \text{ for all } s,t\in S
$$
is called the \textbf{shift action}. If $S$ is bicancellative, the shift action is surjective.

An important special case is the \emph{symbolic} one, that is, when $K$ is a finite \emph{alphabet} $\alphabet$ endowed with the discrete topology. If $|\alphabet| \geq 2$, the space of configurations $\alphabet^S$ is a Cantor space, and together with the action $S\overset{\sigma}{\acts} \alphabet^S$ is called the \textbf{full} $S$\textbf{-shift}. A closed $S$-invariant subset of $ \alphabet^S$ will be called an $S$\textbf{-subshift}. Given alphabets $\alphabet, \mathcal{B}$, and $S$-subshifts $X \subseteq \alphabet^S$ and $Y \subseteq \mathcal{B}^S$, we say that $\varphi: X \to Y$ is a {\bf sliding block code} if there exists a finite subset $F \subseteq S$ and a map $\Phi: \alphabet^F \to \mathcal{B}$ such that $\varphi(x)(s) = \Phi((s \cdot x)|_F)$ for every $x \in X$ and $s \in S$. A sliding block code is always continuous and $S$-equivariant and, due to a straightforward generalization of Curtis-Hedlund-Lyndon theorem, these two properties characterize them. Given a finite subset $F \subseteq S$, the {\bf higher $F$-block code}  will be the particular sliding block code $\phi_F: \alphabet^S \to (\alphabet^F)^S$ given by $\Phi = \mathrm{Id}_{\alphabet^F}$. It is direct to check that $\phi_F$ is injective and a conjugacy between $\alphabet^S$ and $\phi_F(\alphabet^S)$. In particular, $\phi_F(\alphabet^S)$ is an $S$-subshift.

Given subsets $F \subseteq E \subseteq S$ and $w \in \alphabet^E$, we define the {\bf cylinder set induced by $w$ at $F$} as
$$
[w;F] = \{x \in \alphabet^S: x(t) = w(t) \text{ for all } t \in F\}.
$$
The collection of all cylinder sets $[x;F]$ with $x\in \alphabet^S$ and $F$ finite is a basis of clopen sets for the topology of $\alphabet^S$.

\subsection{Periodic configurations on groups}\label{sec:periodic configurations on groups}
Let $G$ be a group and $\alphabet^G$ be the full $G$-shift for some alphabet $\alphabet$. An element $x\in \alphabet^G$ is \textbf{periodic} if it has a finite $G$-orbit, that is, if $|Gx|<\infty$. Equivalently, $x\in \alphabet^G$ is periodic if and only if its \textbf{stabilizer}
$$
\text{Stab}_G(x) = \{g \in G: g \cdot x = x\}
$$
is a finite-index subgroup of $G$, i.e., $[G : \text{Stab}_G(x)] < \infty$. These notions extend to any group action $G\acts X$.

A group $G$ is \textbf{residually finite} if for every pair of distinct elements $g,h\in G$ there is a finite group $Q$ and a group morphism $\theta\colon G\rightarrow Q$ such that $\theta(g)\neq \theta (h)$. Residually finite groups have the following equivalent characterizations: 
    \begin{enumerate}
        \item[\textup{(i)}] $G$ is isomorphic to a subgroup of a Cartesian product of finite groups.
        \item[\textup{(ii)}] The set of periodic points is dense in $\alphabet^G$ for every finite alphabet $\alphabet$.
    \end{enumerate}
    Examples of residually finite groups include $\Z^d$, $\F_d$ and $\mathrm{GL}_d(\Z)$ for any $d\ge 1$. Furthermore, the class of residually finite groups is closed under taking subgroups, Cartesian products and inverse limits. In contrast, the Baumslag--Solitar group $\mathrm{BS}(2,3)$ provides an example of a finitely generated group that is not residually finite. For more details on residually finite groups, we refer the reader to \cite[Chapter 2]{cellular_automata}.

\section{Periodicity for monoid actions}
\label{sec2}

The definition of periodicity in the context of monoid actions is subtle and there may be various notions that capture it. In the case of group actions, due to invertibility, the finite orbit condition is sufficient for capturing the concept comprehensively. However, even in the case of a single non-invertible surjective endomorphism, that is, an action of the monoid $\N$, there could exist finite orbits with a transient initial section, yielding the distinction between periodic and pre-periodic points. 

\subsection{Periodic points and finitely supported invariant measures}

One of our main purposes is to study finitely supported invariant measures. With this goal in mind, we introduce the following definition of periodicity.

\begin{definition}
Let $S\acts X$ be a monoid action. An element $x\in X$ will be called \textbf{pre-periodic} if $|Sx|<\infty$, and \textbf{periodic} if $|Sx|<\infty$ and the set $Sx$ is completely $S$-invariant, that is, $tSx=Sx$ for all $t\in S$. The set of $S$-periodic points of $X$ will be denoted as $\Periodic(X,S)$.
\end{definition}

\begin{example}
For $S=\N$, complete $S$-invariance of $Sx$ implies $|Sx|<\infty$. Note that this is not the case in general. Consider, for instance, the shift action $\N^2\acts \{0,1\}^{\N^2}$ and the element $x\in \{0,1\}^{\N^2}$ given by $x(n,m)=1$ if and only if $n=m$. The $\N^2$-orbit of $x$ is infinite, and every element of $\N^2$ acts bijectively upon its orbit.    
\end{example}

A \textbf{periodic measure} in $M_S(X)$ will be a measure $\mu\in M_S(X)$ such that $\supp(\mu)$ is finite. The following result justifies our chosen definition of periodicity.

\begin{lemma}
    Let $S\acts X$ be a Borel action, and $x\in X$ a pre-periodic point. Then, there is an $S$-invariant measure $\mu\in M_S(X)$ with $\textup{supp}(\mu)=Sx$ if and only if $x$ is periodic.
\end{lemma}

\begin{proof}
    Suppose $x$ is not periodic and choose $t\in S$ such that $t\colon Sx\rightarrow Sx$ is not bijective. Then, since $Sx$ is finite, there exists $y\in Sx$ such that $t^{-1}(y)\cap (Sx)=\varnothing$. Take any $S$-invariant measure $\mu\in M_S(X)$ with $\textup{supp}(\mu)\subseteq Sx$. Then, $\mu(X-Sx)=0$, so 
    $$\mu(\{y\})=\mu(t^{-1}(y))=\mu(t^{-1}(y)\cap Sx)=0.$$
    Thus, $\text{supp}(\mu)\subsetneq Sx$, so no invariant measure can be supported upon the whole of $Sx$. Conversely, assume each $t \in S$ acts as a bijection of $Sx$ and consider the measure given by
    $$\mu(A)=\frac{|A\cap Sx|}{|Sx|} \quad \text{ for every } A\in \mathcal{B}(X).$$
    Define $\varphi_t\colon t^{-1}A\cap Sx\rightarrow A\cap Sx$ by $y\mapsto t\cdot y$. For each $t \in S$, the function $\varphi_t$ is injective as a consequence of $t$ being injective upon $Sx$. Choose any $y\in A\cap Sx$. Since $tSx=Sx$, there exists $y'\in Sx$ with $t\cdot y'=y\in A$, so $y'\in t^{-1}A$ as well, showing surjectivity of $\varphi_t$. Thus, $\mu$ satisfies
    $$\mu(t^{-1}A)=\frac{|t^{-1}A\cap Sx|}{|Sx|}=\frac{|A\cap Sx|}{|Sx|}=\mu(A),$$
    as we wanted.
\end{proof}

In the case of $\N$-actions, the orbit of every pre-periodic point contains a periodic point or, equivalently, the support of an $S$-invariant measure. However, the next example shows that this is not necessarily the case for general monoid actions. 

\begin{example}
\label{exmp:tree}
Consider $\F_2^{+}=\langle a,b\rangle^+$ acting on the full $\F_2^+$-shift $\{0,1\}^{\F_2^+}$. Let $x\in \{0,1\}^{\F_2^{+}}$ be the configuration defined by $x(\varepsilon) = 0$, and by $x(as)=1$ and $x(bs)=0$ for all $s\in \F_2^+$. The orbit of this configuration is $\F_2^{+}x = \{x,a\cdot x\}$. As $b\F_2^{+}x = \{x\}$, $x$ is not periodic, and no subset of this orbit can support an invariant measure (see Figure \ref{bad_orbit_figure}). Although these kind of orbits exhibit a valid form of periodicity, they will not be taken in consideration here.

\begin{figure}[ht]
\centering
\begin{tikzcd}
x\arrow[rr, bend left, "a"]\arrow[loop left, out=-140, in=140, distance=1.5cm,looseness=1,"b"]
& &  \arrow[ll, bend left, "b" {pos=0.47}] a\cdot x \arrow[loop right, out=40, in=-40, distance=1.5cm,looseness=1,"a"]
\end{tikzcd}\hspace{1cm} %
\begin{tikzcd}[column sep=-0.3em,cells={nodes={inner sep=2pt}}]
      &   &   &   &   &   &   & |[draw,circle]| 0\arrow[dllll,dash,"a"'] \arrow[drrrr,dash,"b"] \\
      &   &   & |[draw,circle]| 1\arrow[dll,dash]\arrow[drr,dash] &   &   &   &   &   &   &   & |[draw,circle]| 0 \arrow[dll,dash]\arrow[drr,dash] \\
      & |[draw,circle]|1\arrow[dl,dash]\arrow[dr,dash] &   &   &   & |[draw,circle]|0\arrow[dl,dash]\arrow[dr,dash] &   &   &   & |[draw,circle]|1\arrow[dl,dash]\arrow[dr,dash] &   &   &   & |[draw,circle]|0\arrow[dl,dash]\arrow[dr,dash]\\
    |[draw,circle]|1 &   & |[draw,circle]|0 &   & |[draw,circle]|1 &   & |[draw,circle]|0 &   & |[draw,circle]|1 &   & |[draw,circle]|0 &   & |[draw,circle]|1 &   & |[draw,circle]|0
\end{tikzcd}

\caption{A diagram of the orbit $\F_2^{+}x$ and of $x$ as a labeling of $\F_2^+$.}		
\label{bad_orbit_figure}
\end{figure}
\end{example}

Notice that the action of $\F_2^+$ upon the orbit of the point $x$ defined in Example \ref{exmp:tree} is \textbf{transitive}, i.e., for all $s,t\in S$, $s \cdot x \in S(t \cdot x)$. This holds true for every periodic point.

\begin{proposition}
\label{prop:transitive}
If $x\in X$ is periodic, then the action $S\acts Sx$ is transitive. In particular, two $S$-periodic orbits are either equal or disjoint.
\end{proposition}
	
\begin{proof}
Suppose that $S\acts Sx$ is not transitive. Then, there exists $t \in S$ such that $x\not\in S(t \cdot x)$. Let 
$$k=\min\{i\geq 1:\text{ there is a }n>i\text{ with }t^n\cdot x=t^{i}\cdot x\},$$
which exists because $Sx$ is finite. By definition, there is an $n>k$ with $t^n\cdot x=t^k\cdot x$, and we must have $t^{n-1}\cdot x\neq t^{k-1}\cdot x$. Note that 
$$t(t^{n-1}\cdot x)=t^n\cdot x=t^k\cdot x=t(t^{k-1}\cdot x),$$
contradicting the injectivity of $t$ upon $Sx$. Therefore, $S\acts Sx$ is transitive.

For the final statement, let $x_1,x_2\in X$ be two $S$-periodic elements. If $y\in Sx_1\cap Sx_2$, there are elements $s_1,s_2\in S$ such that $y=s_1\cdot x_1=s_2\cdot x_2$. Since the action $S\acts Sx$ is transitive, we may choose an element $t\in S$ with $ts_1\cdot x_1=x_1$, so $ts_2\cdot x_2=x_1$, which implies $Sx_1\subseteq Sx_2$. By symmetry, we get $Sx_2\subseteq Sx_1$, and so both orbits coincide.
\end{proof}

\begin{remark}
Notice that if $x$ is pre-periodic for an $\N$-action, then the transitivity of $\N \acts Sx$ is equivalent to $x$ being periodic. However, Proposition~\ref{prop:transitive} combined with Example \ref{exmp:tree} show that periodicity is a strictly stronger notion for general monoid actions.
\end{remark}

The set of periodic measures will be denoted by $\Per(X,S)$, and the set of periodic ergodic measures in $M_S(X)$ will be denoted as $\Perg(X,S)$, which coincides with $\Per(X,S)\cap \partial_e M_S(X)$. We have the following characterizations of periodic measures.
	
	\begin{proposition}\label{prop:structure_periodic_measures}
		Let $S\acts X$ be a continuous action and $\mu\in M_S(X)$. Then,  
		\begin{enumerate}
			\item[\textup{(i)}] $\mu\in \Per(X,S)$ if and only if $\supp(\mu)$ is a finite disjoint union of periodic $S$-orbits,
			\item[\textup{(ii)}] $\mu\in \Perg(X,S)$ if and only if $\supp(\mu)$ is a single periodic $S$-orbit.
		\end{enumerate}
		
	\end{proposition}
	
	\begin{proof}
		To prove (i), first assume that $\mu\in \Per(X,S)$. Then, by Lemma~\ref{lemma:surjective_support} we have that $\supp(\mu)$ is $S$-invariant, which means $Sx\subseteq \supp(\mu)$ for all $x\in \supp(\mu)$. If it were the case that an element $x\in \supp(\mu)$ does not belong to the orbit of an element of $\supp(\mu)$, then $s^{-1}(x)\subseteq X-\supp(\mu)$ for every $s\in S$, yielding $\mu(\{x\})=0$, a contradiction.
        This means that $\supp(\mu)=\bigcup_{x\in \supp(\mu)}Sx$, which implies the statement, since periodic orbits are either equal or disjoint by Proposition~\ref{prop:transitive}. The converse is direct.
		
		To prove (ii), first assume $\mu\in \Perg(X,S)$. Then, by part (i) there exist $S$-periodic elements $x_1,\dots,x_m\in X$ such that 
        $$\supp(\mu)=\bigsqcup_{i=1}^mSx_i.$$
        Given $t\in S$ and $i\neq j$, if $x\in t^{-1}(Sx_i)\cap Sx_j$, then $t\cdot x\in Sx_i\cap tSx_j=Sx_i\cap Sx_j=\varnothing$, so $t^{-1}(Sx_i)\cap Sx_j=\varnothing$. Also note that $Sx_i\subseteq t^{-1}(Sx_i)$ for every $t\in S$ and $i$, as $x_i$ is $S$-periodic. Therefore,
		$$\mu(t^{-1}(Sx_i)\symdiff Sx_i)=\mu(t^{-1}(Sx_i)-Sx_i)\leq \mu(X-\supp(\mu))=0.$$
		 Hence, since $\mu$ is ergodic, $\mu(Sx_i) \in \{0,1\}$, so $\supp(\mu)$ consists of a single orbit. 
		
		Conversely, if $\supp(\mu)=Sx$ for an $S$-periodic point $x$ and $A\in \mathcal{B}(X)$ is such that $\mu(s^{-1}A\symdiff A)=0$ for all $s\in S$, then $Sx\cap(A-s^{-1}A)=\varnothing$ for all $s\in S$, and we have two cases. First, if $Sx\cap A=\varnothing$, then $\mu(A)=0$. Second, if $t\cdot x\in Sx\cap A$ for some $t \in S$, necessarily we must have $t\cdot x\in s^{-1}A$ for every $s \in S$ as well, obtaining $st\cdot x\in A$ for every $s\in S$. Since the action of $S$ upon $Sx$ is transitive, this implies $Sx\subseteq A$, and so $\mu(A)=1$. Thus, $\mu$ is ergodic.
	\end{proof}
	
	As a consequence of Proposition~\ref{prop:structure_periodic_measures}, a periodic measure $\mu$ can be written as
	$$\mu=\sum_{i=1}^{m}\frac{\mu(Sx_i)}{|Sx_i|}\sum_{x\in Sx_i}\delta_x,$$
where $\sum_{i=1}^m\mu(Sx_i)=1$, with $m=1$ if and only if the measure is ergodic.

\subsection{Residual finiteness and periodic orbits}

We aim to relate algebraic properties of monoids and both pre-periodic and periodic orbits. In order to do this, let us start by introducing two properties that are monoid analogues to residual finiteness in groups.

    \begin{definition}\label{def: residually}
		A monoid $S$ is \textbf{residually a finite monoid} (resp. \textbf{residually a finite group}) if for every pair of distinct elements $t,t'\in S$ there is a finite monoid (resp. group) $Q$ and a monoid morphism $\theta\colon S\rightarrow Q$ such that $\theta(t)\neq \theta (t')$.
    \end{definition}

    \begin{remark}
		Clearly, being residually a finite group implies being residually a finite monoid. For general monoids, the converse fails. Indeed, every non-bicancellative finite monoid is residually a finite monoid, whereas monoid that are residually a finite group are embeddable (see Proposition~\ref{prop:charac_res_finite_embedding} below). If $S$ is a group, both notions coincide with the classical notion of residual finiteness for groups, as the image of a group via a monoid morphism is always a group.
	\end{remark}

    \begin{remark}
    \label{rem:residually_finite_morphism}
    If $S$ is residually a finite monoid (resp. residually a finite group) and $F\subseteq S$, then for every $t,t'\in F$ there exist a monoid (resp. group) $Q_{t,t'}$ and a monoid morphism $\theta_{t,t'}\colon S\to Q_{t,t'}$ such that $\theta_{t,t'}(t)\ne\theta_{t,t'}(t')$ if $t\neq t'$, and $Q_{t,t'}$ is the trivial group if $t=t'$. Thus, we may define a monoid morphism $\theta_F\colon S\to\prod_{t, t'\in F} Q_{t,t'}$ such that $\theta(t)\ne\theta(t')$ for any $t,t'\in F$, $t\ne t'$, by sending any $s \in S$ to the tuple $(\theta_{t,t'}(s))_{t, t'\in F}$. 
    
    As a consequence of this, $S$ is residually a finite monoid (resp. residually a finite group) if and only if for every finite subset $F\subseteq S$, there is a finite monoid (resp. group) $Q$, namely $\prod_{t, t'\in F} Q_{t,t'}$, and a monoid morphism $\theta\colon S\to Q$, such that $\theta\rvert_{F}$ is injective.
    \end{remark}
	
    The following result is analogous to the characterization of residually finite groups provided in \S\ref{sec:periodic configurations on groups}.
	
	\begin{proposition}
    \label{prop:charac_res_finite_embedding}
		A monoid $S$ is residually a finite monoid (resp.  residually a finite group) if and only if it can be embedded into a Cartesian product of finite monoids (resp. groups). In particular, a monoid is residually a finite group if and only if it can be embedded into a residually finite group.
	\end{proposition}
	
	\begin{proof}
        Assume $S$ is residually a finite monoid (resp. residually a finite group). The map $\theta_Q$ for $Q=S$, following the notation from Remark \ref{rem:residually_finite_morphism}, is an embedding into a product of finite monoids (resp. groups).
	\end{proof}

     In \cite[Proposition~2.14]{CeccheriniCoornaert2015}, the authors prove the equivalence between being residually a finite monoid and having a dense set of points with finite orbit in $\alphabet^S$ for every finite alphabet $\alphabet$. They call these points \emph{periodic}, but they correspond to our definition of pre-periodic points. Our first main result is the following, which in particular shows that periodicity and pre-periodicity characterize different algebraic properties.

	\firstproposition*
	
	\begin{proof}
    The equivalence between being residually a finite monoid and the denseness of the set of pre-periodic points in $\alphabet^S$ is proven in \cite[Proposition~2.14]{CeccheriniCoornaert2015}.
    
    Assume that $S$ is residually a finite  group. Let $x\in\alphabet^S$ be an arbitrary point, and $F\subseteq S$ be a finite subset. Then, by Remark \ref{rem:residually_finite_morphism} there exist a finite group $Q$ and a monoid morphism $\theta\colon S\to Q$ with $\theta\rvert_{F}$ injective. 
        
		Note that every $w\in\alphabet^Q$ defines a point $\bar{w}\in\alphabet^S$ by $\bar{w}(s)=w(\theta(s))$ for all $s\in S$. It holds true that $s\cdot \bar{w} = \overline{\theta(s)\cdot w}$ for any $s\in S$, since for every $t\in S$ we have
        $$(s\cdot \bar{w})(t)=\bar{w}(ts)=w(\theta(ts))=w(\theta(t)\theta(s))=(\theta(s)\cdot w)(\theta(t))=\big{(}\overline{\theta(s)\cdot w}\big{)}(t).$$
        Thus, for any $w\in \alphabet^Q$, the element $\bar{w}\in \alphabet^S$ is pre-periodic. Indeed, the above proven equality shows that $S\bar{w}\subseteq \{\bar{w}':w'\in Qw\}$, which together with the fact that $Q$ is finite prove that $|S\bar{w}|<\infty$. Since $Q$ is a group, $\theta(S)$ is a group as well. Thus, given $t\in S$, for every $s \in S$, there is an $s'\in S$ such that $\theta(t)\theta(s')=\theta(s)$, yielding
        $$ts'\cdot \bar{w}=\overline{\theta(ts')\cdot w}=\overline{\theta(s)\cdot w}=s\cdot \bar{w}.$$
        Thus, $tS\bar{w}=S\bar{w}$ for every $t \in S$, and it follows that $\bar{w}$ is periodic.
        
        Define a configuration $w_x\colon Q\to\alphabet$ by $w_x(\theta(s))=x(s)$ for all $s\in F$, extending it arbitrarily to the rest of $Q$ if needed. As the map $\theta\rvert_F\colon F\to Q$ is injective, this is well-defined. The corresponding point $\overline{w_x}$ in $\alphabet^S$ is periodic and satisfies the equality $\overline{w_x}\rvert_F=x\rvert_F$. This shows that, for any $x\in\alphabet^S$ and any finite $F\subseteq S$, the cylinder $[x;F]$ must contain a periodic point, and thus the set of periodic points is dense. 

        We proceed to prove the converse implication. First, assume $\alphabet^S$ has a dense set of periodic points. Let $s,s'\in S$ be distinct elements. Choose two distinct elements from $\alphabet$, which, for simplicity, will be denoted by $0$ and $1$, and a periodic element $x\in \alphabet^S$ such that $x(s)=0$ and $x(s')=1$. Consider the finite group $Q=\mathrm{Sym}(Sx)$, and define a monoid morphism $\theta\colon S\rightarrow Q$ by sending $a\in S$ to the function $f_a\colon Sx\rightarrow Sx$ given by $f_a(b\cdot x)=ab\cdot x$ for all $b\in S$. The definition of periodicity implies that $f_a\in \mathrm{Sym}(Sx)$ for each $a\in S$. By our choice of $x$, we must have $(s\cdot x)(1_S)=x(s)=0$ and $(s'\cdot x)(1_S)=x(s')=1$. Hence, $s\cdot x \ne s'\cdot x$, and thus $f_s$ and $f_{s'}$ correspond to different elements of $Q$, that is, $\theta(s)\neq\theta(s')$. We conclude that $S$ is residually a finite group.
	\end{proof}



\subsection{The periodic approximation property}
In view of Proposition~\ref{thm:A}, it is natural to ask under what conditions upon the acting monoid $S$ the periodic measures on $\alphabet^S$ are weak-* dense. Considering this, we now introduce a definition which will be fundamental throughout this work.

\begin{definition}
    A monoid $S$ has the \textbf{periodic approximation property}, or simply the PA \textbf{property}, if $\Per(\alphabet^S,S)$ is weak-* dense in $M_S(\alphabet^S)$ for every finite alphabet $\alphabet$. A monoid $S$ has the \textbf{ergodic periodic approximation property}, or simply the EPA \textbf{property}, if $\Perg(\alphabet^S,S)$ is weak-* dense in $M_S(\alphabet^S)$ for every finite alphabet $\alphabet$. 
\end{definition}

This definition extends the one discussed for groups in the Introduction. In the case of groups, the PA property is equivalent to the property MD described by Kechris \cite{kechris2020weak,kechris2012weakcontainment} in the context of weak containment of group actions. It is known that all amenable, residually finite groups have the EPA property \cite{ren2018}. Similarly, Bowen showed that free groups on finitely many generators have the PA property \cite{bowen2003periodicity}. Examples of residually finite groups without the PA property include the special linear group $\operatorname{SL}_n(\Z)$ for $n\ge 3$ \cite{kechris2012weakcontainment} (in contrast to $\operatorname{SL}_2(\Z)$, which has the PA property) and $\F_2\times\F_2$ as a consequence of the negative answer to the Connes' Embedding Problem (see \cite[p. 27]{kechris2020weak} and \cite{ji2021}).

We stress that, although the PA property has been attentively studied for groups, this has not been the case with general monoids. As a first preliminary result, we show that residual finiteness turns out to be a key necessary condition for having the PA property, as seen below.

\begin{proposition}\label{prop:PA_are_residually_finite}
    Let $S$ be a bicancellative monoid. If $S$ has the PA property, then it is residually a finite group.
\end{proposition}

\begin{proof}
    Assume $S$ is not residually a finite group. By Proposition~\ref{thm:A}, there is  $x\in\alphabet^S$ and a finite subset $F\subseteq S$ such that $[x;F]$ does not contain a periodic point. Choose a measure $\mu\in M_S(\alphabet^S)$ such that $\mu([x;F])>0$ (take, for instance, a fully supported Bernoulli measure), and let $(\mu_n)_n$ be a sequence in $\Per(\alphabet^S,S)$ that weak-* converges to $\mu$. By Proposition~\ref{prop:structure_periodic_measures}, we can assume that for every $n\geq 1$,
    $$\mu_n=\sum_{i=1}^{m_n}\frac{\mu_n(Sx^n_i)}{|Sx^n_i|}\sum_{y\in Sx^n_i}\delta_y,$$
    where each $x^n_i$ is periodic for every $1\leq i\leq m_n$. By weak-* convergence, we have that $\mu_n([x;F])\to\mu([x;F])$. However, for all $n\geq 1$ and $1\leq i\leq m_n$, the set $[x;F]\cap Sx^n_i$ is empty, meaning that $\mu_n([x;F])=0$, a contradiction with our choice of $\mu$.
\end{proof}


\section{Measure-theoretical invertible extensions and the PA property}
\label{sec3}

We want to understand the PA property in the more general landscape of monoids, taking advantage of what is already known for groups. With this goal in mind, we restrict ourselves to the class of embeddable monoids. The purpose of this section is to establish a direct connection between the PA property of an embeddable monoid $S$ and the PA property of the corresponding free $S$-group. In order to do this, the tool of choice will be the natural extension construction, which associates to an $S$-action a corresponding invertible action of the free $S$-group which extends the original action in a natural way.

\subsection{Topological natural extensions}\label{section: topological natural extension}

In \cite{bricenobustosdonoso1}, topological natural extensions were profusely discussed. Based on this work, we consider the following definition.

\begin{definition}
\label{def:natural_extension}
Let $S$ be an embeddable monoid, $S \overset{\alpha}{\curvearrowright} X$ a continuous action and $\receivingSgroup=(G, \eta)$ a receiving $S$-group. A \textbf{topological $\receivingSgroup$-extension} of $\alpha$ is a tuple $\hat{\mathbf{X}}=(\hat{X}, \beta, \tau)$, where $\beta$ is a continuous action $G \curvearrowright \hat{X}$ and $\tau: \hat{X} \rightarrow X$ is a topological $\eta$-extension. The action $\alpha$ is said to be $\textbf{topologically } \G$\textbf{-extensible} when it admits a topological $\receivingSgroup$-extension.

A $\receivingSgroup$-extension $\hat{\mathbf{X}}$ of $\alpha$ is said to be \textbf{natural} if for any other $\receivingSgroup$-extension $\hat{\mathbf{X}}'=\left(\hat{X}', \beta', \tau'\right)$ of $\alpha$ there is a unique $G$-equivariant continuous function $\varphi\colon \hat{X}' \rightarrow \hat{X}$ satisfying $\tau \circ \varphi=\tau'$.
\end{definition}

Given a receiving $S$-group $\receivingSgroup=(G, \eta)$ and a continuous action $S \curvearrowright X$, consider the set
$$
X_{\receivingSgroup}=\left\{\left(x_h\right)_{h \in G} \in X^G: s \cdot x_h=x_{\eta(s) h} \text { for all } s \in S \text { and } h \in G\right\},
$$
endowed with the subspace topology of the product topology, and the projection $\pi\colon X_{\receivingSgroup} \rightarrow X$ at $1_G$. Let $G \curvearrowright X_{\receivingSgroup}$ be the restriction of the shift action $G \stackrel{\sigma}{\curvearrowright} X^G$. Since $X_{\receivingSgroup}$ is closed in $X^G$, and $X^G$ is compact---as we are assuming $X$ is a compact metric space---, we have that $X_{\receivingSgroup}$ compact as well. In addition, the map $\pi$ is continuous and $\eta$-equivariant in the sense that $s\cdot \pi(\overline{x}) = \pi(\eta(s) \cdot \overline{x})$ for all $\overline{x}\in X_{\receivingSgroup}$ and $s\in S$.

\begin{proposition}[{\cite[Proposition~2.4]{bricenobustosdonoso1}}]
\label{prop:characterizations_of_extensibility}
Let $S \overset{\alpha}{\curvearrowright} X$ be a continuous $S$-action and $\receivingSgroup$ a receiving $S$-group. The following are equivalent.
\begin{itemize}
\item[(i)] The map $\pi\colon X_{\receivingSgroup} \rightarrow X$ is surjective.
\item[(ii)] The natural $\receivingSgroup$-extension of $\alpha$ exists.
\item[(iii)] There exists a $\receivingSgroup$-extension of $\alpha$.
\end{itemize}
When these conditions hold, the natural $\receivingSgroup$-extension of $\alpha$ is isomorphic to $\left(X_{\receivingSgroup}, \sigma, \pi\right)$.
\end{proposition}

The following result characterizes the free $S$-group in the residually finite case in terms of topological extensibility of finite actions.

\begin{proposition}
\label{prop:extensibility_residually_finite}
Let $S$ be an embeddable monoid, let $\freeSgroup=(\Gamma,\gamma)$ be its free $S$-group, and let $\receivingSgroup$ be a receiving $S$-group. If $\receivingSgroup$ is a realization of the free $S$-group, then every finite surjective action $S \acts X$ is topologically $\receivingSgroup$-extensible. Moreover, if $\Gamma$ is residually finite, the converse holds.
\end{proposition}

\begin{proof}
    Assume $\receivingSgroup=\freeSgroup$, and let $S\acts X$ be an action by surjective maps on a finite set. We assume without loss of generality that the topology on $X$ is discrete, as a topological $\freeSgroup$-extension also extends the same system with a coarser topology, which is then topologically $\freeSgroup$-extensible by Proposition~\ref{prop:characterizations_of_extensibility}. 
    
    Because the action $S\acts X$ is by surjective maps and $X$ is finite, $S$ acts by bijective maps on $X$. Consider the monoid homomorphism $\eta\colon S\to \mathrm{Sym}(X)$ associated to this action. Then $(\mathrm{im}(\eta),\eta)$ is an $S$-group, so the universal property of the free $S$-group gives a group homomorphism $\theta\colon \Gamma\to \mathrm{im}(\eta)\leq \mathrm{Sym}(X)$ satisfying $\theta\circ \gamma=\eta$. Consider the action $\Gamma\acts X$ given by $g\cdot x=\theta(g)(x)$ for all $x\in X$ and $g\in \Gamma$. As the identity map $\mathrm{id}_X\colon X\to X$ is surjective and satisfies $\mathrm{id}_X(\theta(\gamma(s))(x))=s\cdot x$ for all $s\in S$ and $x\in X$, the action $S\acts X$ is a $\gamma$-factor of $\Gamma\acts X$. Thus $S\acts X$ is $\freeSgroup$-extensible by Proposition~\ref{prop:characterizations_of_extensibility}.
    
    The fact that the converse holds when $\Gamma$ is residually finite is immediate from \cite[Proposition~3.8]{bricenobustosdonoso1}.
\end{proof}

\subsection{Measure-theoretical extension-compatibility}\label{sec: MEC}
Inspired by the definition of topological $\receivingSgroup$-extension, we introduce the following analogue in the p.m.p. context.

\begin{definition}\label{definition:invertible_extension}
Let $S$ be an embeddable monoid, $S\overset{\alpha}{\acts} (X,\mu)$ a p.m.p. $S$-action, and $\receivingSgroup = (G,\eta)$ a receiving $S$-group. A \textbf{p.m.p.} $\receivingSgroup$-\textbf{extension} of $\alpha$ is a tuple $(\hat{X},\hat{\alpha},\hat{\mu},\genextmap)$, where $\hat{\alpha}$ is a p.m.p. $G$-action $G\acts(\hat{X},\hat{\mu})$ and $\genextmap\colon (\GextensionSet,\hat{\mu})\rightarrow (X,\mu)$ is an $\eta$-factor. In this case we say that $\hat{\alpha}$ \textbf{p.m.p.}  $\receivingSgroup$\textbf{-extends} $\alpha$, and that $\alpha$ is \textbf{p.m.p.} $\receivingSgroup$\textbf{-extensible}.
\end{definition}

The following result allows us to reduce the study of p.m.p. $\receivingSgroup$-extensibility to the case of shift actions on the Cantor space.

\begin{proposition}\label{prop:symbolic_representation_measure}
    For every p.m.p. action $S\overset{\alpha}{\acts}(X,\mu)$ there exists $\nu\in M_S(C^S)$, where $C$ is the Cantor space, such that $\alpha$ is conjugate to the p.m.p. shift action $S \overset{\sigma}{\acts} (C^S,\nu)$.
\end{proposition}

\begin{proof}
First note that there must exist an injective Borel map $\psi\colon X\to C$ that maps Borel sets to Borel sets. Indeed, $X$ is either finite, countably infinite or $|X|=|\R|$, so there is an injective Borel map $X\to \R$, which by \cite[Theorem~15.1]{kechris2012classical} sends Borel sets to Borel sets. One can then compose this map with an arbitrary Borel isomorphism $\R\to C$ to obtain $\psi$.

Define $\varphi\colon X\to C^S$ by $\varphi(x)(t)=\psi(t\cdot x)$ for all $x\in X$ and $t\in S$. The map $\varphi$ is Borel, since $\pi_t\circ \varphi=\psi\circ \alpha_t$, where $\pi_t\colon C^S\to C$ is the $t$-coordinate projection and $\psi$ and $\alpha_t$ are Borel for all $t\in S$. It is also clear that $\psi$ is injective.

For a Borel subset $A\subseteq X$, note that $\psi(A)$ is Borel and
$$\varphi(A)=\{y \in C^S:y(1_S)\in \psi(A)\text{ and }s\cdot \psi^{-1}(y(t))=\psi^{-1}(y(st))\text{ for all } s,t\in S\}.$$
Indeed, if $y \in C^S$ is such that $y(1_S)\in\psi(A)$ and $s\cdot \psi^{-1}(y(t))=\psi^{-1}(y(st))$ for all $s,t\in S$, then $y(1_S)=\psi(x)$ for some $x\in A$. Then $\varphi(x)(s)=\psi(s\cdot x)=\psi(s\cdot \psi^{-1}(y(1_S)))=y(s)$ for all $s\in S$, so $y\in\varphi(A)$. Conversely, if $y\in \varphi(A)$ there is an $x\in A$ such that for every $t\in S$ we have $y(t)=\psi(t\cdot x)$. Then $y(1_S)\in\psi(A)$, and for every $s,t\in S$, $s\cdot \psi^{-1}(y(t))=st\cdot x=\psi^{-1}(y(st))$.

Through the above description of $\varphi(A)$, we have that
$$\varphi(A)=\pi_{1_S}^{-1}(\psi(A))\cap\bigcap_{s,t \in S} \{y\in C^S: (\alpha_s \circ \psi^{-1}\circ \pi_t)(y) = (\psi^{-1}\circ\pi_{st})(y)\},$$
and therefore $\varphi(A)$ is a Borel subset of $C^S$, as all maps involved are Borel measurable. We conclude that $\varphi\colon X\to \varphi(X)$ is a Borel isomorphism and $\varphi(X)$ is measurable.

Finally, $\varphi$ is $S$-equivariant, since for every $x\in X$ and $s,t\in S$ we have
$$\varphi(s\cdot x)(t)=\psi(t\cdot(s\cdot x))=\psi(ts\cdot x)=\varphi(x)(ts)=[s\cdot \varphi(x)](t).$$
Hence, letting $\nu=\varphi_*\mu$, $\varphi\colon (X,\mu)\to (C^S,\nu)$ becomes a conjugacy. 
\end{proof}


Given $\mu\in M_S(C^S)$ and a receiving $S$-group $\receivingSgroup=(G,\eta)$, the homomorphism $\eta\colon S\to G$ induces a topological $\eta$-factor map $\eta^*\colon C^G\to C^S$ by setting $\eta^*(x)\mapsto x\circ \eta$ for all $x\in C^G$. The map $\eta^*$, in turn, induces a push-forward map $(\eta^*)_*\colon M(C^G)\to M(C^S)$, defined by $(\eta^*)_*\bar{\mu}(A)=\bar{\mu}((\eta^*)^{-1}A)$ for all $A\in \mathcal{B}(C^S)$ and $\bar{\mu}\in M(C^G)$. The fact that $\eta^*$ is $\eta$-equivariant implies that $(\eta^*)_*$ is well-defined as a map $M_G(C^G)\to M_S(C^S)$. It is standard that $(\eta^*)_*$ is weak-* continuous (see, e.g., \cite[Theorem~15.14(1)]{aliprantis2006}). Considering this, we further reduce the study of the p.m.p. $\receivingSgroup$-extensibility of $S\acts (C^S,\mu)$ to studying measures in $M_G(C^G)$.

\begin{proposition}\label{prop:reduction_to_the_natural_extension}
    Let $S$ be an embeddable monoid, $\receivingSgroup=(G,\eta)$ be a receiving $S$-group, $K$ be a Polish space and let $\mu\in M_S(K^S)$. Then, $S\acts(K^S,\mu)$ is p.m.p. $\receivingSgroup$-extensible if and only if there exists $\bar{\mu}\in M_G(K^G)$ such that $(\eta^*)_*\bar{\mu}=\mu$.
\end{proposition}

\begin{proof}
    If there exists $\bar{\mu}\in M_G(K^G)$ such that $(\eta^*)_*\bar{\mu}=\mu$, then $(K^G,\sigma,\bar{\mu},\eta^*)$ is a p.m.p. $\receivingSgroup$-extension of $S\acts (K^S,\mu)$.

    Conversely, let $(\GextensionSet,\hat{\alpha},\hat{\mu},\genextmap)$ be a  p.m.p. $\receivingSgroup$-extension of $S\acts (K^S,\mu)$. Write $\tau\colon \hat{X}_0\to X_0$, where $X_0\subseteq K^S$ is a $\mu$-conull $S$-invariant set and $\hat{X}_0\subseteq \hat{X}$ is a $\hat{\mu}$-conull $G$-invariant set witnessing that $\tau$ is an $\eta$-factor. The formula
    $$\varphi(\hat{x})(h)=\tau(h\cdot \hat{x})(1_S)\quad \text{for }\hat{x}\in\hat{X}_0\text{ and }h\in G,$$
    allows us to define a Borel map $\varphi\colon \hat{X}\to K^G$ such that,  on $\hat{X}_0$, is $G$-equivariant and satisfies $\eta^*\circ \varphi=\tau$. Now, for all $A\in \mathcal{B}(K^S)$ we have
    $$\mu(A)=\mu(A\cap X_0)=\hat{\mu}(\tau^{-1}(A\cap X_0))=\hat{\mu}(\varphi^{-1}((\eta^*)^{-1}(A))\cap \hat{X}_0)=(\varphi_*\hat{\mu})((\eta^*)^{-1}(A)).$$
    Hence, taking $\bar{\mu}=\varphi_*\hat{\mu}$ we get $\bar{\mu}\in M_G(K^G)$ and $(\eta^*)_*\bar{\mu}=\mu$.
\end{proof}

\begin{proposition}\label{prop:isomorphisms}
Let $\theta\colon S\to S'$ be a homomorphism of monoids and $\bar{\theta}\colon\receivingSgroup\to\receivingSgroup'$ be a morphism between a receiving $S$-group $\receivingSgroup = (G,\eta)$ and a receiving $S'$-group $\receivingSgroup' =(G',\eta')$ satisfying $\bar{\theta}\circ \eta=\eta'\circ \theta$. Let $S\acts (X,\mu)$ and $S'\acts (X',\mu')$ be p.m.p. actions, and $\varphi\colon(X',\mu')\to (X,\mu)$ be a $\theta$-factor between them. If $S'\acts (X',\mu')$ is p.m.p. $\receivingSgroup'$-extensible, then $S\acts (X,\mu)$ is p.m.p. $\receivingSgroup$-extensible.
\end{proposition}

\begin{proof}
By Proposition~\ref{prop:symbolic_representation_measure}, we may assume that $X=C^S$ and $X'=C^{S'}$. Since $S'\acts (C^{S'},\mu')$ is p.m.p. $\receivingSgroup'$-extensible, by Proposition~\ref{prop:reduction_to_the_natural_extension} there exists $\bar{\mu}'\in M_{G'}(C^{G'})$ such that $((\eta')^*)_*\bar{\mu}'=\mu'$. Let $X'_0\subseteq C^{S'}$ be a $\mu'$-conull set, $X_0\subseteq C^{S}$ be a $\mu$-conull set, and $\varphi\colon X'_0\to X_0$ be a $\theta$-factor map. Define a map $\bar{\varphi}\colon C^{G'}\to C^G$ by
\[\bar{\varphi}(\bar{x}')(h)=\left\{
\begin{array}{ll}
    [(\varphi\circ(\eta')^*)(\bar{\theta}(h)\cdot \bar{x}')](1_S) & \text{if }\bar{x}'\in ((\eta')^*)^{-1}(X'_0) \\
    0^G & \text{otherwise}.
\end{array}\right.\]
The map $\bar{\varphi}$ is measurable, as for all $h\in G$ the map $\pi_h\circ \bar{\varphi}$ is equal on $((\eta')^*)^{-1}(X'_0)$ to the composition of the map $\pi_{1_S}\circ \varphi\circ (\eta')^*$ and the action of $\bar{\theta}(h)$, and constant elsewhere. It is also clear that $\bar{\varphi}$ is $\bar{\theta}$-equivariant. Moreover, the identity \(\eta^*\circ \bar{\varphi}=\varphi\circ (\eta')^*\) holds on $((\eta')^*)^{-1}(X'_0)$, as for every $\bar{x}'\in C^{G'}$ and $s\in S$ we have
\begin{align*}
    [(\eta^*\circ \bar{\varphi})(\bar{x}')](s)&=\bar{\varphi}(\bar{x}')(\eta(s))\\
    &=[(\varphi\circ(\eta')^*)(\bar{\theta}(\eta(s))\cdot \bar{x}')](1_S)\\
    &=[\varphi((\eta')^*(\eta'(\theta(s))\cdot \bar{x}'))](1_S)\\
    &=[\varphi(\theta(s)\cdot (\eta')^*(\bar{x}'))](1_S)\\
    &=\varphi((\eta')^*(\bar{x}'))(s),
\end{align*}
where the second to last inequality follows from the following computation:
\begin{align*}
    (\eta')^*(\eta'(\theta(s))\cdot \bar{x}')(s')&=(\eta'(\theta(s))\cdot \bar{x}')(\eta'(s'))\\
    &=\bar{x}'(\eta'(s'\theta(s)))\\
    &=((\eta')^*(\bar{x}'))(s'\theta(s))\\
    &=(\theta(s)\cdot (\eta')^*(\bar{x}'))(s').
\end{align*}

Setting $\bar{\mu}:=\bar{\varphi}_*\bar{\mu}'$, it follows that
\[(\eta^*)_*\bar{\mu}=(\eta^*\circ \bar{\varphi})_*\bar{\mu}'=(\varphi\circ (\eta')^*)_*\bar{\mu}'=\varphi_*\mu'=\mu.\]
\end{proof}

We introduce the following class of monoids.

\begin{definition}
    Let $S$ be an embeddable monoid. We say that $S$ is \textbf{measure-theoretically extension-compatible (MEC)} if every p.m.p. action $S\acts (X,\mu)$ is p.m.p. $\receivingSgroup$-extensible for some receiving $S$-group $\receivingSgroup$.
\end{definition}

A direct consequence of Proposition~\ref{prop:isomorphisms} is that if $S\acts (X,\mu)$ is $\receivingSgroup$-extensible for some receiving $S$-group, then it is $\freeSgroup$-extensible, where $\freeSgroup$ is the free $S$-group. Thus, $S$ is MEC if and only if every p.m.p. action of $S$ is $\freeSgroup$-extensible. We point out that if $\receivingSgroup$ is not the free $S$-group $(\Gamma,\gamma)$ and $\Gamma$ is residually finite, by Proposition~\ref{prop:characterizations_of_extensibility} there is a surjective action $S\acts X$ on a finite set that is not topologically $\receivingSgroup$-extensible. This follows from \cite[Proposition~3.8]{bricenobustosdonoso1}, which in fact asserts that $X_\receivingSgroup=\varnothing$. Therefore, Proposition~\ref{prop:reduction_to_the_natural_extension} implies that $S\acts (X,\mu)$ is not p.m.p. $\receivingSgroup$-extensible for any choice of $\mu\in M_S(X)$.

\begin{proposition}\label{prop: characterization of MEC}
 The following are equivalent:
 \begin{enumerate}
     \item[(i)] $S$ is MEC (measure-theoretically extension-compatible).
     \item[(ii)] $(\gamma^*)_*\colon M_\Gamma(C^\Gamma) \to M_S(C^S)$ is surjective.
     \item[(iii)] $(\gamma^*)_*\colon M_\Gamma(\alphabet^\Gamma) \to M_S(\alphabet^S)$ is surjective for every finite alphabet $\alphabet$.
 \end{enumerate}
\end{proposition}

\begin{proof}
    It is clear from Proposition~\ref{prop:reduction_to_the_natural_extension} that (i) implies (iii). Also, Proposition~\ref{prop:symbolic_representation_measure} combined with Proposition~\ref{prop:reduction_to_the_natural_extension} immediately give that (ii) implies (i). It just remains to prove that (iii) implies (ii).

    Identify $C$ with $\{0,1\}^\N$ and, for each $n \in \N$, let $r_n\colon C \to C$ be the map defined by
    \[
    (r_n((c_i)_i))_k =
    \begin{cases}
    c_k &   \text{if } k \leq n,   \\
    0   &   \text{if } k > n.
    \end{cases}
    \]
    Set $\alphabet_n := r_n(C)$, which is finite. Denote by $r_n^S\colon C^S \to \alphabet_n^S$ the map that applies $r_n$ coordinatewise, which is measurable and $S$-equivariant.

    Assume (iii). Take $\mu \in M_S(C^S)$ and set $\mu_n := (r_n^S)_*\mu \in M_S(\alphabet_n^S)$ for each $n \in \N$. By (iii) and Proposition~\ref{prop:reduction_to_the_natural_extension}, there exists $\bar{\mu}_n \in M_\Gamma(\alphabet_n^\Gamma)$ such that $(\gamma^*)_*\bar{\mu}_n = \mu_n$. As $\alphabet_n^\Gamma \subseteq C^\Gamma$, each $\bar{\mu}_n$ can be regarded as an element of $M_\Gamma(C^\Gamma)$. By compactness of $M_\Gamma(C^\Gamma)$, the sequence $(\bar{\mu}_n)_n$ has an accumulation point $\bar{\mu}$. Without loss of generality, we can assume that $\bar{\mu}_n \to \bar{\mu}$.

    Since $r_n^S(x)\to x$ for every $x\in C^S$, we have that $f(r_n^S(x)) \to f(x)$ for any $f \in C(C^S)$. By the DCT,
    \[
    \int{f}d\mu_n = \int{f \circ r_n^S}d\mu \to \int{f}d\mu,
    \]
    so $\mu_n\to\mu$ in the weak-* topology. Therefore, by continuity of $(\gamma^*)_*$,
    \[
    (\gamma^*)_*\bar{\mu} = \lim_{n \to\infty} (\gamma^*)_*\bar{\mu}_n =\lim_{n\to\infty}\mu_n =\mu,
    \]
    so $(\gamma^*)_*\colon M_\Gamma(C^\Gamma) \to M_S(C^S)$ is surjective.
\end{proof}

\subsection{Periodicity and extensibility}

Recall the construction $\pi\colon X_{\receivingSgroup}\to X$ introduced in \S\ref{section: topological natural extension}. Given a receiving $S$-group $\receivingSgroup=(G,\eta)$ and a continuous action $S\acts X$, we define
\[\Ext_{\receivingSgroup}(X,S):=\mathrm{im}(\pi_*)=\{\mu\in M_S(X):\text{ there exists }\bar{\mu}\in M_G(X_{\receivingSgroup})\text{ with }\pi_*\bar{\mu}=\mu\}.\]
We will sometimes refer to elements of $\Ext_{\receivingSgroup}(X,S)$ as $\receivingSgroup$\textbf{-extensible measures}.

\begin{remark}\label{rem: Ext extensibility}
    We point out that $\mu\in \Ext_{\receivingSgroup}(X,S)$ if and only if the associated p.m.p. action $S\acts (X,\mu)$ is p.m.p. $\receivingSgroup$-extensible. Its truth is easily seen in the case $X=K^S$ by Proposition~\ref{prop:reduction_to_the_natural_extension} and the identification between $(K^S)_{\receivingSgroup}$ and $K^G$.
\end{remark}

\begin{proposition}
\label{prop:ext_weak_closed}
Let $\receivingSgroup$ be a receiving $S$-group and $S\acts X$ be a continuous action. Then $\Ext_\receivingSgroup(X,S)$ is a convex subset of $M_S(X)$. If $X$ is compact, then $\Ext_{\receivingSgroup}(X,S)$ is weak-* closed.
\end{proposition}

\begin{proof}
    The set $\Ext_\receivingSgroup(X,S)$ is convex, as it is the image of the convex set $M_G(X_\receivingSgroup)$ under the affine map $\pi_*$. Assume now that $X$ is compact, and write $\receivingSgroup=(G,\eta)$. Let $\mu\in M_S(X)$, and let $(\mu_n)_n$ be a sequence in $\Ext_\receivingSgroup(X,S)$ converging to $\mu$. Consider, for every $n\in \N$, a measure $\bar{\mu}_n\in M_G(X_\receivingSgroup)$ such that $\pi_*\bar{\mu}_n=\mu_n$. Since $X_\receivingSgroup$ is compact, $M_G(X_\receivingSgroup)$ is weak-* compact. Thus, we can find a subsequence $(\bar{\mu}_{n_k})_{k}$ of $(\bar{\mu}_n)_n$ weak-* converging to some $\bar{\mu}\in M_G(X_\receivingSgroup)$. By continuity of $\pi_*$, we get 
    $$\mu=\lim_{k\to \infty}\mu_{n_k}=\lim_{k\to\infty}\pi_*\bar{\mu}_{n_k}=\pi_*\left(\lim_{k\to \infty}\bar{\mu}_{n_k}\right)=\pi_*\bar{\mu}.$$
    Therefore, $\mu\in \Ext_\receivingSgroup(X,S)$.
\end{proof}

A first key observation is that periodic measures are always $\freeSgroup$-extensible when $\freeSgroup$ is the free $S$-group.

\begin{proposition}\label{per_are_ext}
    Assume that $S$ is an embeddable monoid, that $\freeSgroup=(\Gamma,\gamma)$ is a realization of the free $S$-group, and let $S\acts X$ be a continuous action. Then, for every $\mu \in \Perg(X,S)$, there exists $\bar{\mu} \in \Perg(X_{\freeSgroup},\Gamma)$ such that $\pi_*\bar{\mu} = \mu$. In particular, $\Per(X,S) \subseteq \Ext_{\freeSgroup}(X,S)$.
\end{proposition}

\begin{proof}
Let $\mu\in \Perg(X,S)$ be given by
    $$\mu=\frac{1}{|Sx|}\sum_{y\in Sx}\delta_{y} \quad \text{ for some } x\in \Periodic(X,S).
    $$

    Consider a disjoint copy $\hat{S}=\{\hat{s}:s\in S\}$ of $S$, with the operation $\hat{t}\hat{s}=\hat{st}$ for $\hat{s},\hat{t}\in \hat{S}$. We define an action of $\hat{S}$ on $Sx$ as follows: for every $\hat{s} \in \hat{S}$ and $y\in Sx$, we let $\hat{s} \cdot y$ to be the only element of $s^{-1}(\{y\}) \cap Sx$. Notice that this is indeed an action, because $\hat{st}\cdot y$ is the only element of 
    $$(st)^{-1}(\{y\})\cap Sx=t^{-1}(s^{-1}(\{y\}))\cap Sx=t^{-1}(s^{-1}(\{y\})\cap Sx)\cap Sx,$$
    and the only element of the latter is exactly $\hat{t}\cdot (\hat{s}\cdot y)$.
    
    Since both $S$ and $\hat{S}$ act upon $Sx$, we immediately get an action $S\ast \hat{S}\acts Sx$ by concatenation, where $S\ast \hat{S}$ denotes the free product monoid, where the empty word $\varepsilon$ acts as the identity permutation.
    
We want to see the action of $S\ast \hat{S}$ descends to an action of $\Gamma$. To do this, we define, for every $\hat{s}\in \hat{S}$, the element $\hat{\gamma}(\hat{s})=\gamma(s)^{-1}\in \Gamma$. Note that this defines a morphism $\hat{\gamma}\colon \hat{S}\to \Gamma$, since $\hat{\gamma}(\hat{t}\hat{s})=\hat{\gamma}(\hat{st})=\gamma(st)^{-1}=\gamma(t)^{-1}\gamma(s)^{-1}=\hat{\gamma}(\hat{t})\hat{\gamma}(\hat{s})$ for all $\hat{s},\hat{t}\in \hat{S}$.
This induces a morphism $\gamma_*\colon S\ast \hat{S}\to \Gamma$ by 
$$\gamma_*(s_1\hat{t}_1\cdots s_n\hat{t}_n)=\gamma(s_1)\hat{\gamma}(\hat{t}_1) \cdots \gamma(s_n)\hat{\gamma}(\hat{t}_n).$$

Given $w_1,w_2 \in S \ast \hat{S}$, declare $w_1 \sim w_2$ if and only if $w_1$ and $w_2$ induce the same element in $\mathrm{Sym}(Sx)$, i.e., if $w_1 \cdot y = w_2 \cdot y$ for every $y \in Sx$. We want to check that, if $\gamma_*(w_1)=\gamma_*(w_2)$, then $w_1 \sim w_2$.

First, observe that if $w_1 \sim w_2$, then for any $u,v \in S \ast \hat{S}$,
\[
u w_1 v \cdot y = u \cdot (w_1 \cdot (v \cdot y)) = u \cdot (w_2 \cdot (v \cdot y)) = u w_2 v \cdot y \quad \text{ for every } y \in Sx,
\]
thus making $\sim$ a congruence on $S \ast \hat{S}$. Next, consider the quotient monoid $S \ast \hat{S}/\sim$. Take an arbitrary element $w = s_1\hat{t}_1\cdots s_n\hat{t}_n$ in $S \ast \hat{S}$ and define $\hat{w} = t_n\hat{s}_n\cdots t_1\hat{s}_1$. Since $s\hat{s}\sim \varepsilon \sim \hat{s}s$ for every $s \in S$, inductively we have $w\hat{w} \sim \varepsilon \sim \hat{w}w$. This means 
$$
[w]_\sim[\hat{w}]_\sim=[w\hat{w}]_\sim=[\varepsilon]_\sim = [\hat{w}w]_\sim =[\hat{w}]_\sim[w]_\sim.
$$
Hence, $G:= S \ast \hat{S}/\sim$ is actually a group, where the inverse of $[w]_{\sim}$ is just $[\hat{w}]_{\sim}$. 
  
  Let $q\colon S\ast \hat{S}\to G$ be the quotient map associated to $\sim$. The map $q\rvert_{S} \colon S\rightarrow G$ is a monoid morphism, as equal elements of $S$ define equal elements in $\textup{Sym}
  (Sx)$. Analogously, $q\rvert_{\hat{S}}\colon \hat{S}\rightarrow G$ is a monoid morphism, too. Since $[s_1\hat{t}_1\cdots s_n\hat{t}_n]_{\sim}=[s_1]_{\sim}[t_1]_{\sim}^{-1}\cdots [s_n]_{\sim}[t_n]_{\sim}^{-1}$ for every $s_1\hat{t}_1\cdots s_n\hat{t}_n\in S\ast \hat{S}$, we find that $(G,q\rvert_{S})$ is an $S$-group (note that, in general, the morphism $q\rvert_{S}$ may not be injective). By the universal property of the free $S$-group, we get a group morphism $\theta\colon \Gamma\rightarrow G$ such that $\theta\circ \gamma=q\rvert_{S}$, and it is verified that $\theta\circ \hat{\gamma}=q\rvert_{\hat{S}}$ as well. Therefore, $\theta\circ \gamma_*=q$, which implies $w_1\sim w_2$ whenever $\gamma_*(w_1)=\gamma_*(w_2)$, as desired.  
  
  We can now define an action $\Gamma\acts Sx$ by setting $\gamma_*(w)\cdot y=w\cdot y$ for every $w\in S\ast \hat{S}$ and $y\in Sx$, as a consequence of the fact that $\gamma_*(S\ast \hat{S})=\Gamma$. For a given $y\in Sx$, let $\bar{y}=(y_h)_{h\in G}\in X_{\freeSgroup}$ be given by $y_h=h\cdot y$ for every $h\in \Gamma$. Note that if $g\in \Gamma$, then
    $$g\cdot \bar{x}=g\cdot (h\cdot x)_{h\in G}=(hg\cdot x)_{h\in G}=(h\cdot (g\cdot x))_{h\in G}=\overline{g\cdot x}.$$
    This shows $\Gamma\bar{x}\subseteq\{\bar{y}:y\in Sx\}$. Also, if $y\in Sx$, then $y=s\cdot x$ for some $s\in S$, which implies $\bar{y}=\overline{s\cdot x}=\overline{\gamma(s)\cdot x}=\gamma(s)\cdot \bar{x}$. Notice that in this last equation there are three different actions involved. Thus, $\{\bar{y}:y\in Sx\}\subseteq \Gamma\bar{x}$. Since $\bar{y}\neq \bar{y}'$ whenever $y,y'\in Sx$ differ, we have that $|\Gamma\bar{x}|=|\{\bar{y}:y\in Sx\}|=|Sx|$. Thus, $\bar{x}$ is $\Gamma$-periodic, and the measure
    $$\bar{\mu}=\frac{1}{|\Gamma\bar{x}|}\sum_{\bar{y}\in \Gamma\,\bar{x}}\delta_{\bar{y}}=\frac{1}{|Sx|}\sum_{\bar{y}\in \Gamma\,\bar{x}}\delta_{\bar{y}}$$
    is $\Gamma$-periodic and ergodic, as it is supported on a single orbit. Finally, note that $$\bar{\mu}(\pi^{-1}A)=\frac{1}{|Sx|}\sum_{\bar{y}\in \Gamma\,\bar{x}}\delta_{\bar{y}}(\pi^{-1}A)=\frac{1}{|Sx|}\sum_{\bar{y}\in \Gamma\,\bar{x}}\delta_{\pi(\bar{y})}(A)=\frac{1}{|Sx|}\sum_{y\in Sx}\delta_{y}(A)=\mu(A)$$
    for every $A\in \mathcal{B}(X)$, so $\pi_{*}\overline{\mu}=\mu$. Hence $\mu \in \textup{Ext}_{\freeSgroup}(X,S)$.

    Finally, as every $\mu \in \Per(X,S)$ is a convex combination of elements in $\Perg(X,S)$, $\pi_*$ is linear, and, by Proposition~\ref{prop:ext_weak_closed}, $\Ext_{\freeSgroup}(X,S)$ is a convex subset of $M_S(X)$, we conclude.
\end{proof}

Note that if we remove the assumption that $\freeSgroup$ is the free $S$-group then Proposition~\ref{per_are_ext} does not necessarily hold, as the set $X_\receivingSgroup$ might be empty for an arbitrary receiving $S$-group $\receivingSgroup$ \cite[Proposition~3.8]{bricenobustosdonoso1}.

The following result concerns the behaviour of images of periodic $G$-orbits under $S$-equivariant maps.
	
	\begin{lemma}\label{image_of_periodic}
		Let $\receivingSgroup=(G,\eta)$ be a receiving $S$-group, and let $G\acts \hat{X}$ and $S\acts X$ be two actions. If $\phi:\hat{X}\rightarrow X$ is an $\eta$-equivariant map and $\hat{x}\in\hat{X}$ is $G$-periodic (i.e., $|G\hat{x}|<\infty$), then $\phi(\hat{x})$ is $S$-periodic and $\phi(G\hat{x})=S\phi(\hat{x})$.
	\end{lemma}
	
	\begin{proof}
		If $s\in S$, it is clear that $\phi(\eta(s)\cdot \hat{x})=s\cdot \phi(\hat{x})\in S\phi(\hat{x})$. Also, $\eta(s)$ has finite order as an element of $\textup{Sym}(G\hat{x})$, so there exists $n>1$ such that $\eta(s^n)$ acts as the identity of $G\hat{x}$. Thus, 
        $$\phi(\eta(s)^{-1}\cdot \hat{x})=\phi(\eta(s^{n-1})\cdot\hat{x})=s^{n-1}\cdot\phi(\hat{x})\in S\phi(\hat{x}).$$
        Since $G=\langle\eta(S)\rangle$, we conclude that $\phi(G\hat{x})\subseteq S\phi(\hat{x})$, which directly implies that $\phi(G\hat{x})= S\phi(\hat{x})$. 
        
        It is now clear that $\phi(\hat{x})$ is pre-periodic. To see that $S\phi(\hat{x})$ is completely $S$-invariant, let $s,t\in S$. Since the set $\{\eta(t)^{n}\eta(s) \cdot \hat{x}:n\geq 1\}$ is finite, there are $k>j$ such that $\eta(t)^k\eta(s) \cdot \hat{x}=\eta(t)^j\eta(s) \cdot \hat{x}$. The action on $\hat{X}$ is given by a group, so $\eta(s) \cdot \hat{x}=\eta(t)^{k-j}\eta(s) \cdot \hat{x}$ and $t^{k-j}\in S$. Then, due to the $\eta$-equivariance of $\phi$,
		$$s \cdot \phi(\hat{x})=t \cdot (t^{k-j-1}s \cdot \phi(\hat{x}))\in tS\phi(\hat{x}),$$
		so $t\colon S\phi(\hat{x})\rightarrow S\phi(\hat{x})$ is surjective, and hence bijective by finiteness of $S\phi(\hat{x})$.

	\end{proof} 

\begin{proposition}\label{supp_lemma}
Let $S \acts X$ be a continuous action on a compact metric space, and $\receivingSgroup$ be a receiving $S$-group. If $\bar{\mu}\in M_G(X_{\receivingSgroup})$, then $\supp(\pi_*\bar{\mu})= \pi(\supp(\bar{\mu}))$. In particular, we have the following:
    \begin{enumerate}
        \item[(i)] $\pi_{*}(\Perg(X_{\receivingSgroup},G))\subseteq \Perg(X,S)$ and $\pi_{*}(\Per(X_{\receivingSgroup},G))\subseteq \Per(X,S)$.
        \item[(ii)] If $\overline{\Per(X_{\receivingSgroup},G)}=M_G(X_{\receivingSgroup})$, then $\Ext_{\receivingSgroup}(X,S)\subseteq \overline{\Per(X,S)}$.
        \item[(iii)] If $\overline{\Perg(X_{\receivingSgroup},G)}=M_G(X_{\receivingSgroup})$, then $\Ext_{\receivingSgroup}(X,S)\subseteq \overline{\Perg(X,S)}$.
    \end{enumerate}
\end{proposition}

\begin{proof}
First, we prove that $\supp(\pi_*\bar{\mu})= \pi(\supp(\bar{\mu}))$. Let $x\in \supp(\pi_*\bar{\mu})$ and fix a decreasing sequence $(U_n)_n$ of open neighborhoods of $x$ with $\bigcap_{n}U_n=\{x\}$. Since $x\in \supp(\pi_*\bar{\mu})$, $\pi_*\bar{\mu}(U_n)>0$ for every $n$.
In particular, for every $n$, $\pi^{-1}(U_n)\cap \supp(\bar{\mu})$ is non-empty, so there exists $\bar{x}_n\in \supp(\bar{\mu})$ with $\pi(\bar{x}_n)\in U_n$. By compactness of $X_\receivingSgroup$, we can take a subsequence $(\bar{x}_{n_k})_k$ converging to $\bar{x}\in \supp(\bar{\mu})$ as $k\to \infty$. By continuity of $\pi$, we get $\pi(\bar{x}_{n_k})\to \pi(\bar{x})$, but since $\pi(\bar{x}_{n_k})\in U_{n_k}$, we also have that $\pi(\bar{x}_{n_k})\to x$. Thus, we conclude that $\pi(\bar{x})=x$, so $x\in \pi(\supp(\bar{\mu}))$. For the opposite inclusion, let $\bar{x}\in \supp(\bar{\mu})$. If $U$ is an open neighborhood of $\pi(\bar{x})$, then $\bar{x}\in\pi^{-1}U$ and $\pi_*\bar{\mu}(U)=\bar{\mu}(\pi^{-1}U)>0$, so $\pi(x)\in \supp(\pi_*\bar{\mu})$, following the desired inclusion.

To prove (i), let $\bar{\mu}\in \Per(X_{\receivingSgroup},G)$. There exist elements $x_1,\dots, x_m\in \Periodic(X_\receivingSgroup,G)$ such that $\supp(\bar{\mu})=Gx_1\cup \cdots \cup Gx_m$, so by the equality proven above and Lemma~\ref{image_of_periodic}, we have
$\supp(\pi_*\bar{\mu})=\bigcup_{i=1}^{m}\pi(Gx_i)=\bigcup_{i=1}^{m}S\pi(x_i)$. Taking $m=1$ proves the ergodic case.

To prove (ii) and (iii), by (i) and weak-* continuity of $\pi_*$ we have $\Ext_{\receivingSgroup}(X,S)=\im(\pi_*)=\pi_*\bigl(\overline{\Per(X_{\receivingSgroup},G)}\bigr)\subseteq \overline{\Per(X,S)}$,
and the same argument shows that $\Ext_{\receivingSgroup}(X,S)\subseteq \overline{\Perg(X,S)}$.
\end{proof}

\begin{proposition}\label{prop: identification between periodic and extensible measures}
Let $S$ be an embeddable monoid, and let $\freeSgroup=(\Gamma, \gamma)$ be a realization of the free $S$-group. If $S \curvearrowright X$ is a continuous action on a compact metric space and $P\left(X_{\freeSgroup}, \Gamma\right)$ (resp. $\partial_eP\left(X_{\freeSgroup}, \Gamma\right)$) is weak-* dense in $M_{\Gamma}\left(X_{\freeSgroup}\right)$, then
$$
\operatorname{Ext}_{\freeSgroup}(X, S)=\overline{P(X, S)} \quad\left(\text{ resp. } \operatorname{Ext}_{\freeSgroup}(X, S)=\overline{\partial_eP(X, S)}\right).
$$  
\end{proposition}

\begin{proof}
    We prove the periodic case; the ergodic periodic case is identical. By Proposition~\ref{per_are_ext}, $\Per(X,S)\subseteq \Ext_{\freeSgroup}(X,S)$. By Proposition~\ref{supp_lemma} (ii), $\Ext_{\freeSgroup}(X,S) \subseteq \overline{\Per(X,S)}$. By Proposition~\ref{prop:ext_weak_closed}, $\Ext_{\freeSgroup}(X,S)$ is weak-* closed, so $\Ext_{\freeSgroup}(X,S) = \overline{\Per(X,S)}$.
\end{proof}

 The next theorem shows a direct connection between the PA property and measure-theoretical extension-compatibility.

\theoremB*

\begin{proof}
    Let $A$ be a finite alphabet. Proposition~\ref{per_are_ext}, the weak-* closedness of $\Ext_{\freeSgroup}(A^S,S)$, and the identification between $(A^S)_{\freeSgroup}$ and $A^\Gamma$ together imply that $(\gamma^*)_*\colon M_\Gamma(A^\Gamma)\to M_S(A^S)$ is surjective. By Proposition~\ref{prop: characterization of MEC}, $S$ is MEC. Assume now that $\Gamma$ has the PA (resp. EPA) property and that $S$ is MEC. By Proposition~\ref{prop: characterization of MEC}, $\mathrm{Ext}_{\freeSgroup}(A^S,S)=M_S(A^S)$ for every finite alphabet, hence Proposition \ref{prop: identification between periodic and extensible measures} implies that $S$ has the PA (resp. EPA) property.
\end{proof}

\begin{remark}
    Note that, as a consequence of Theorem~\ref{thm:characterization of the PA property}, if $\Gamma$ has the EPA property, then $S$ has the EPA property if and only if $S$ has the PA property.
\end{remark}

\begin{remark}
Due to Proposition~\ref{prop:PA_are_residually_finite}, a necessary condition for an embedabble monoid to have the PA property is to be residually a finite group. The situation where $S$ is residually a finite group and $\Gamma$ is not residually finite can happen. An example of this would be the Baumslag--Solitar monoid
$$\textup{BS}(2,3)^+=\langle a,b\mid ab^2=b^3a\rangle^+,$$
which is residually a finite group \cite[Theorem~4.5]{jackson2002decision}, while the corresponding Baumslag-Solitar group $\textup{BS}(2,3)$ is known to be non-Hopfian, and hence non-residually finite \cite{baumslag1962some}. This leaves open the possibility for a monoid $S$ to have the PA property, while $\Gamma$ does not.
\end{remark}

\section{Applications and examples}
\label{sec4}

In this section, we study measure-theoretical extension-compatibility and periodic approximation for several families of monoids. We establish MEC in the left reversible and free cases and use Theorem~\ref{thm:characterization of the PA property} to obtain the PA or EPA property whenever the corresponding property holds for the free $S$-group.

\subsection{The left reversible case}

In the left reversible case, we have the following extensibility
result for p.m.p. actions.

\begin{theorem}
\label{thm:extensibility_reversible}
Every left reversible bicancellative monoid is MEC.
\end{theorem}

\begin{remark}
While this work was in preparation, versions of Theorem~\ref{thm:extensibility_reversible} appearing in \cite[Theorem~2.9]{farhangi1} and \cite[Theorem~A.2]{rodriguez2025} were brought to our attention by Sohail Farhangi. For completeness, we include our own proof. A previous, prototypical version was also included in the Master's thesis of the third author \cite[Theorem~2.7.7]{donoso2024}.
\end{remark}

\begin{proof}
    Let $\freeSgroup_S=(\Gamma_S,\gamma)$ be the group of right fractions of $S$. By Proposition~\ref{prop: characterization of MEC}, it suffices to show that for every finite alphabet $\alphabet$ and measure $\mu\in M_S(\alphabet^S)$, there exists $\bar{\mu}\in M_{\Gamma_S}(\alphabet^{\Gamma_S})$ such that $(\gamma^*)_*\bar{\mu}=\mu$. Choose $\bar{\mu}_0\in M(\alphabet^{\Gamma_S})$ satisfying $(\gamma^*)_*\bar{\mu}_0=\mu$, using the surjectivity of $(\gamma^*)_*\colon M(\alphabet^{\Gamma_S}) \to M(\alphabet^S)$ \cite[Theorem~15.14(5)]{aliprantis2006}. Fix a sequence $(F_n)_{n\geq 1}$ of finite subsets of $\Gamma_S$ such that $F_n\subseteq F_{n+1}$ for all $n\geq 1$ and $\Gamma_S=\bigcup_{n\geq 1}F_n$. By \cite[Lemma~2.22]{bricenobustosdonoso1}, there exists a sequence $(g_n)_{n\geq 1}$ in $\Gamma_S$ such that $F_ng_n\subseteq \gamma(S)$. Writing $g_n=\gamma(s_n)\gamma(t_n)^{-1}$ with $s_n,t_n\in S$, we observe that $F_n\gamma(s_n)\subseteq \gamma(S)\gamma(t_n)\subseteq\gamma(S)$.
    
    Define $\bar{\mu}_n := (\gamma(s_n))_*\bar{\mu}_0$. By compactness of $M(\alphabet^{\Gamma_S})$, the sequence of measures $(\bar{\mu}_n)_{n\geq 1}$ has a limit point $\bar{\mu}\in M(\alphabet^{\Gamma_S})$, and we may assume that $\bar{\mu}_n \to \bar{\mu}$. By continuity of $(\gamma^*)_*$, $\gamma$-equivariance of $\gamma^*$, and $S$-invariance of $\mu$,
    \[
    (\gamma^*)_*\bar{\mu} = (\gamma^*)_* \lim_{n \to \infty} \bar{\mu}_n  = \lim_{n \to \infty} (\gamma^* \circ \gamma(s_n))_*\bar{\mu}_0 = \lim_{n \to \infty} (s_n \circ \gamma^*)_*\bar{\mu}_0 = \lim_{n \to \infty} (s_n)_*\mu = \lim_{n \to \infty} \mu = \mu.
    \]

    It remains to prove that $\bar{\mu}$ is $\Gamma_S$-invariant. Fix $g\in \Gamma_S$, a finite subset $F \subseteq \Gamma_S$ and $\bar{x}\in\alphabet^{\Gamma_S}$. Let $m\geq 1$ be large enough so that $F\gamma(s_m)\subseteq \gamma(S)$, and write $F\gamma(s_m)=\gamma(K_m)$, with $K_m\subseteq S$. Now, choose $n\geq 1$ large enough so that $\gamma(s_m)^{-1}g\gamma(s_n)$ belongs to $\gamma(S)$, and write $\gamma(t_n)=\gamma(s_m)^{-1}g\gamma(s_n)$ with $t_n\in S$. Then,
\begin{align*}
    \bar{\mu}_0([g\gamma(s_n)^{-1}\cdot\bar{x};Fg\gamma(s_n)])&=\bar{\mu}_0([\gamma(s_n)^{-1}g^{-1}\cdot\bar{x};(F\gamma(s_m))\gamma(s_m)^{-1}g\gamma(s_n)])\\&=\mu([\gamma^*\bigl(\gamma(s_n)^{-1}g^{-1}\cdot\bar{x}\bigr);K_mt_n])\\
        &=\mu(t_n^{-1}[\gamma^*\bigl(\gamma(s_m)^{-1}\cdot\bar{x}\bigr);K_m])\\ &=\mu([\gamma^*\bigl(\gamma(s_m)^{-1}\cdot\bar{x}\bigr);K_m]),
\end{align*}
and therefore
    \begin{align*}
        g_*\bar{\mu}([\bar{x};F])&=\bar{\mu}([g^{-1}\cdot\bar{x};Fg])\\
        &=\lim_{n\to\infty}(\gamma(s_n))_*\bar{\mu}_0([g^{-1}\cdot\bar{x};Fg])\\
        &=\lim_{n\to\infty}\bar{\mu}_0([\gamma(s_n)^{-1}g^{-1}\cdot\bar{x};Fg\gamma(s_n)])\\        &=\mu([\gamma^*\bigl(\gamma(s_m)^{-1}\cdot\bar{x}\bigr);K_m])\\
        &=\bar{\mu}_0([\gamma(s_m)^{-1}\cdot \bar{x};F\gamma(s_m)])\\
        &=\gamma(s_m)_*\bar{\mu}_0([\bar{x};F]).
    \end{align*}
    Taking $m\to\infty$, we conclude that $g_*\bar{\mu}=\bar{\mu}$. Thus, $\bar{\mu}$ is $\Gamma_S$-invariant.
\end{proof}

A direct consequence of Theorem~\ref{thm:extensibility_reversible} and Theorem~\ref{thm:characterization of the PA property} is the following.

\begin{corollary}
    Let $S$ be a left reversible bicancellative monoid, and let $\freeSgroup_S=(\Gamma_S,\gamma)$ be its group of right fractions. If $\Gamma_S$ has the PA (resp. EPA) property, then $S$ has the PA (resp. EPA) property.
\end{corollary}

The following result was proven in \cite{ren2018}.

\begin{theorem}[{\cite[Theorem~1.1]{ren2018}}]\label{thm:ren}
Let $G$ be an infinite discrete countable residually finite amenable group acting on a compact metric space $X$ with specification property. Then $\Perg(X, G)$ is dense in $M_G(X)$ in the weak-* topology.
\end{theorem}

Since the shift action on the full $G$-shift $\alphabet^G$ have the specification property, this result implies that infinite residually finite amenable groups have the EPA property.

A bicancellative monoid $S$ is {\bf left amenable} if there exists a \textbf{left F\o lner sequence}, namely a sequence $(F_n)_n$ of nonempty finite subsets of $S$ such that
$$
\lim_{n\to\infty}\frac{|sF_n \symdiff F_n|}{|F_n|}=0 \quad \text{ for every } s \in S.
$$
See \cite{argabright1967semigroups,frey1960studies,namioka1964folner} for further details. 

\begin{proposition}
\label{prop:epa_property_amenable}
Let $S$ be an infinite left amenable monoid that is residually a finite group, and let $\freeSgroup=(\Gamma,\gamma)$ be its free $S$-group. Then $\Gamma$ has the EPA property. 
\end{proposition}

\begin{proof}
    By \S \ref{rem:unique_receiving_group}, $\freeSgroup$ is isomorphic to $\GrRightFrac{S}=(\Gamma_S,\gamma)$. In particular, $\Gamma\simeq \Gamma_S$. It is well-known that if $S$ is a bicancellative left amenable monoid, then $S$ is left reversible and $\Gamma_S$ is amenable (see \cite[Lemma~1]{Donnelly0} and \cite[Proposition~4.2]{bricenobustosdonoso1}).
    
    Since $S$ is residually a finite group, by Proposition~\ref{prop:charac_res_finite_embedding} we may assume $S\leq H$, where $H$ is a residually finite group. Hence, the subgroup of $H$ generated by $S$, together with the natural embedding $\iota\colon S\to H$, is a receiving $S$-group and thus isomorphic to the group of right fractions of $S$; see \S \ref{rem:unique_receiving_group}. As a consequence, $\Gamma_S$ is isomorphic to a subgroup of the residually finite group $H$ and is thus residually finite as well. Applying Theorem~\ref{thm:ren} to full shifts over finite alphabets, we conclude that $\Gamma\simeq \Gamma_S$ has the EPA property.
\end{proof}

\subsection{The free case}

Fix a positive integer $d$ and let $B:=\{a_1,\dots,a_d\}$. Let $\F_d:=\F(B)$ and $\F_d^+:=\F(B)^+$ be the free group and free monoid on $B$, respectively. The free $\F_d^+$-group is $\freegroup=(\F_d,\iota)$ with $\iota\colon \F_d^+\to \F_d$ the inclusion. We identify $\F_d^+$ with its image $\iota(\F_d^+)$.

\begin{proposition}\label{prop:extensibility_free_case}
For every positive integer $d$, the free monoid $\F_d^+$ is MEC.
\end{proposition}

To prove Proposition~\ref{prop:extensibility_free_case}, we first introduce Markov measures on free monoids and free groups. Fix a finite alphabet $\alphabet$. 
Let $\Sigma\in \{B,B\sqcup B^{-1}\}$ and $T=\langle \Sigma\rangle^+$, so that $T$ is either $\F_d^+$ or $\F_d$. Let $\text{Cay}(T,\Sigma)$ be the left Cayley graph of $T$ with respect to $\Sigma$. The neighborhood of radius $r$ with center on the identity $\varepsilon$ will be denoted by $N_r$. A finite subset $F\subseteq T$ will be called a \textbf{tree} if it induces a connected subgraph $(F,E(F))$ in $\text{Cay}(T,\Sigma)$. The \textbf{root} of a tree $F$ will be the unique element $r(F)\in F$ with minimal distance to $\varepsilon$.

A \textbf{Markov $\Sigma$-tree chain} is a pair $(\mathbf{p},\mathbf{P})$ consisting of a positive probability vector $\mathbf{p}$ (i.e., $\mathbf{p} > \mathbf{0}$ and $\sum_{\ell\in\mathcal{A}}\mathbf{p}_\ell=1$) and a family $\mathbf{P} = (\mathbf{P}^a)_{a \in \Sigma}$ of real $|\mathcal{A}|\times |\mathcal{A}|$ stochastic matrices $\mathbf{P}^a$ (i.e., $\mathbf{P}^a \geq \mathbf{0}$ and $\sum_{\ell\in\mathcal{A}}\mathbf{P}^a_{k,\ell}=1$ for all $k\in\mathcal{A}$).

\begin{definition}\label{Markov_measure_def}
 Given a Markov $\Sigma$-tree chain $(\mathbf{p},\mathbf{P})$, a measure $\mu\in M(\mathcal{A}^{T})$ is said to be \textbf{$(\mathbf{p},\mathbf{P})$-Markov} if, for every $x\in \mathcal{A}^{T}$ and every finite tree $F \subseteq T$ with root $\varepsilon$, we have
$$
\mu([x\, ;F])=\mathbf{p}_{x(\varepsilon)}\cdot \prod_{(t,at)\in E(F)}\mathbf{P}^{a}_{x(t),x(at)}.$$
A measure $\mu\in M(\mathcal{A}^T)$ is said to be \textbf{Markov} if it is $(\mathbf{p},\mathbf{P})$-Markov for some Markov $\Sigma$-tree chain $(\mathbf{p},\mathbf{P})$.
\end{definition}

\begin{lemma}[{\cite[p. 240]{georgii}}]\label{prop:georgii}
For every Markov $\Sigma$-tree chain $(\mathbf{p},\mathbf{P})$, there exists a unique $(\mathbf{p},\mathbf{P})$-Markov measure on $\mathcal{A}^T$. Moreover, this measure is $T$-invariant if and only if 
\begin{enumerate}
\item $\mathbf{p}$ is a left eigenvector of each $\mathbf{P}^a$, i.e., $\mathbf{p}\mathbf{P}^a=\mathbf{p}$ for every $a\in \Sigma$, and
\item if $a, a^{-1} \in \Sigma$, then $\mathbf{p}_k\mathbf{P}^{a^{-1}}_{k,\ell} = \mathbf{p}_\ell\mathbf{P}^a_{\ell,k}$ for every $k,\ell \in \mathcal{A}$.
\end{enumerate}
\end{lemma}

If $(\mathbf{p},\mathbf{P})$ is a Markov $\Sigma$-tree chain such that the induced $(\mathbf{p},\mathbf{P})$-Markov measure is $T$-invariant, we say that $T\overset{\sigma}{\acts} (\mathcal{A}^T,\mu)$ is the \textbf{p.m.p. action induced by} $(\mathbf{p},\mathbf{P})$. 

For the rest of the subsection, we set $S:=\F_d^+$ and $\Gamma:=\F_d$.

\begin{lemma}
\label{lemma:Ext_Fd}
For every Markov $B$-tree chain $(\mathbf{p},\mathbf{P})$ that induces an $S$-invariant measure $\mu \in M(\mathcal{A}^{S})$, there exists a Markov measure $\bar{\mu} \in M_{\Gamma}(\mathcal{A}^{\Gamma})$ such that $\Gamma \overset{\sigma}{\acts} (\mathcal{A}^{\Gamma},\bar{\mu})$ p.m.p. $\freegroup_d$-extends $S \overset{\sigma}{\acts} (\mathcal{A}^{S},\mu)$.
\end{lemma}

\begin{proof}
Given $a\in B$, define $\hat{\mathbf{P}}^a = \mathbf{P}^a$, and let $\hat{\mathbf{P}}^{a^{-1}}$ be defined coordinate-wise as
$$
\hat{\mathbf{P}}^{a^{-1}}_{k,\ell} = \frac{\mathbf{p}_\ell}{\mathbf{p}_k}\mathbf{P}^a_{\ell,k} \quad \text{ for } k,\ell \in \mathcal{A}.
$$
Consider the family $\hat{\mathbf{P}}=\{\hat{\mathbf{P}}^{a}:a\in B\sqcup B^{-1}\}$. First, let us check that $(\mathbf{p},\hat{\mathbf{P}})$ is a Markov $(B \sqcup B^{-1})$-tree chain. For each $a\in B$, we have that $\hat{\mathbf{P}}^{a}$ is stochastic, and for all $k\in \mathcal{A}$,
$$\sum_{\ell\in \mathcal{A}}\hat{\mathbf{P}}^{a^{-1}}_{k,\ell}=\sum_{\ell\in \mathcal{A}}\frac{\mathbf{p}_\ell}{\mathbf{p}_k}\mathbf{P}^{a}_{\ell,k}=\frac{1}{\mathbf{p}_k}\sum_{\ell\in\mathcal{A}}\mathbf{p}_\ell\mathbf{P}^{a}_{\ell,k}=\frac{1}{\mathbf{p}_k}\mathbf{p}_k=1,$$
as $\mathbf{p}\mathbf{P}^{a}=\mathbf{p}$, showing that $\hat{\mathbf{P}}^{a^{-1}}$ is stochastic as well.

Let us check that $(\mathbf{p},\hat{\mathbf{P}})$ induces an $\Gamma$-invariant Markov measure. Condition (2) in Lemma \ref{prop:georgii} is satisfied by $(\mathbf{p},\hat{\mathbf{P}})$ by construction, so it suffices to check that $\mathbf{p}\hat{\mathbf{P}}^{a^{-1}}=\mathbf{p}$ for $a\in B$. Indeed, for every $\ell\in \mathcal{A}$, as $\mathbf{P}^{a}$ is stochastic,
\[\big{(}\mathbf{p}\hat{\mathbf{P}}^{a^{-1}}\big{)}_\ell=\sum_{k\in\mathcal{A}}\mathbf{p}_k\hat{\mathbf{P}}^{a^{-1}}_{k,\ell}=\sum_{k\in\mathcal{A}}\mathbf{p}_k\frac{\mathbf{p}_\ell}{\mathbf{p}_k}\mathbf{P}^{a}_{\ell,k}=\mathbf{p}_\ell\sum_{k\in\mathcal{A}}\mathbf{P}^{a}_{\ell,k}=\mathbf{p}_\ell.\]

Let $\bar{\mu}$ be the $\Gamma$-invariant $(\mathbf{p},\hat{\mathbf{P}})$-Markov measure on $\mathcal{A}^{\Gamma}$ guaranteed by Lemma \ref{prop:georgii}. It remains to show that $\Gamma \overset{\sigma}{\acts} (\mathcal{A}^{\Gamma},\bar{\mu})$ p.m.p. $\freegroup_d$-extends $S \overset{\sigma}{\acts} (\mathcal{A}^{S},\mu)$. To see this, it suffices to check that $(\iota^*)_*\bar{\mu}([x;F]) = \mu([x;F])$ for any finite tree $F \subseteq S$ and $x\in \mathcal{A}^{S}$, where $\iota^*\colon \mathcal{A}^{\Gamma}\rightarrow \mathcal{A}^{S}$ is the pullback map $x\mapsto x\circ \iota= x|_{S}$ (see Proposition \ref{prop:reduction_to_the_natural_extension}). Notice that, due to $S$-invariance of $(\iota^*)_*\bar{\mu}$ and $\mu$, we can assume that $F$ has root $\varepsilon$. Then, it is direct that $(\iota^*)_*\bar{\mu}= \mu$, since both $\mu$ and $\bar{\mu}$ assign the same measure to every cylinder $[x;F]$---as presented in Definition \ref{Markov_measure_def}---, because $\mathbf{P}^a$ and $\hat{\mathbf{P}}^a$ coincide for $a \in B$.
\end{proof}

We want to use Lemma \ref{lemma:Ext_Fd} to show that every p.m.p. action $S\overset{\sigma}{\acts}(\mathcal{A}^{S},\mu)$ is p.m.p. $\freegroup_d$-extensible. To do this, we consider appropriate Markov approximations and recodings of these p.m.p. actions. Recall that the notion of a sliding block code was introduced in \S \ref{sec:shift_actions}.

\begin{proof}[Proof of Proposition \ref{prop:extensibility_free_case}]

Fix $\mu\in M_{S}(\mathcal{A}^{S})$ and let $X=\text{supp}(\mu)$, which is an $S$-subshift of $\mathcal{A}^S$. For each $m \geq 0$, we abbreviate by $\phi_m$ the higher $B_m$-block code $\phi_{B_m}: \mathcal{A}^S \to (\mathcal{A}^{B_m})^S$. Fix $m$, define $\left.X\right\vert_{B_m}=\{x\rvert_{B_m}:x\in X\}$, which is a finite set and will play the role of an alphabet. Set for each $\alpha,\beta\in \left.X\right\vert_{B_m}$ and $a \in \Sigma$,
    $$
    \mathbf{p}_\alpha=\mu([\alpha\,;B_{m}]) \quad \text{ and } \quad \mathbf{P}^{a}_{\alpha,\beta}=
        \dfrac{\mu\big{(}[\alpha\,;B_{m}]\cap a^{-1}[\beta\,;B_{m}]\big{)}}{\mathbf{p}_\alpha}.
        $$
Notice that $\mathbf{p}_\alpha > 0$, since $[\alpha\,;B_{m}] \cap \text{supp}(\mu) \neq \emptyset$. It is clear that $\mathbf{p}$ is a probability vector, and for each $\alpha\in \left.X\right\vert_{B_m}$,
    $$\sum_{\beta\in \left.X\right\vert_{B_m}}\mathbf{P}^{a}_{\alpha,\beta}=\frac{1}{\mathbf{p}_\alpha}\sum_{\beta\in \left.X\right\vert_{B_m}}\mu\big{(}[\alpha\,;B_{m}]\cap a^{-1}[\beta\,;B_{m}]\big{)}=1,$$
    since the sets $\big{\{}a^{-1}[\beta;B_{m}]:\beta\in \left.X\right\vert_{B_m}\big{\}}$ form a partition of $(\left.X\right\vert_{B_m})^{S}$. Thus, $\mathbf{P}^{a}$ is a stochastic matrix for all $a \in \Sigma$. For each $\beta\in \left.X\right\vert_{B_m}$, we get
    $$\sum_{\alpha\in \left.X\right\vert_{B_m}}\mathbf{p}_\alpha \mathbf{P}^{a}_{\alpha,\beta}=\sum_{\alpha\in \left.X\right\vert_{B_m}}\mu\big{(}[\alpha;B_{m}]\cap a^{-1}[\beta;B_{m}]\big{)}=\mu\big{(} a^{-1}[\beta\,;B_{m}]\big{)}=\mathbf{p}_\beta,$$
    so $\mathbf{p}\mathbf{P}^{a}=\mathbf{p}$. Finally, if $a,a^{-1} \in \Sigma$, then
    $$
    \mathbf{p}_\alpha\mathbf{P}^{a}_{\alpha,\beta} = \mu\big{(}[\alpha\,;B_{m}]\cap a^{-1}[\beta\,;B_{m}]\big{)} = \mu\big{(}a[\alpha\,;B_{m}]\cap [\beta\,;B_{m}]\big{)} = \mathbf{p}_\beta\mathbf{P}^{a^{-1}}_{\beta,\alpha},
    $$
    and, by Lemma \ref{prop:georgii}, we conclude $\mathbf{p}$ and $\{\mathbf{P}^{a}:a \in \Sigma\}$ define an $S$-invariant Markov measure $\nu_m$ on $(\left.X\right\vert_{B_m})^{S}$. 
    
    We want to see that $\supp(\nu_m)\subseteq \phi_m(\mathcal{A}^S)$. Notice that, if $y\in (\left.X\right\vert_{B_m})^{S}-\phi_m(\mathcal{A}^S)$, then there exist $s\in S$ and $a \in \Sigma$ such that $[y(s)\,;B_{m}]\cap a^{-1}[y(as)\,;B_{m}]=\varnothing$. Indeed, if we assume otherwise, for every $s\in S$ and $a \in \Sigma$, there is an $x\in [y(s);B_{m}]\cap a^{-1}[y(as);B_{m}]$, so for all $\ell<m$ and $t\in B_\ell$,
    $$y(s)(ta)=x(ta)=(a\cdot x)(t)=y(as)(t).$$ 
    Define $z\in \mathcal{A}^{S}$ by $z(s)=y(s)(\varepsilon)$. Applying iteratively the identity just proven implies that, for $s\in S$ and $t\in B_{m}$,
    $$(\varphi_m(z)(s))(t)=(s\cdot z)(t)=z(ts)=y(ts)(\varepsilon)=y(s)(t),$$
    so $\varphi_m(z)=y$, contradicting that $y\not\in\phi_m(\mathcal{A}^S)$. Hence, $[y(s)\,;B_{m}]\cap a^{-1}[y(as)\,;B_{m}]=\varnothing$ for some $s\in S$ and $a \in \Sigma$, which directly implies that $\nu_m([y\, ;\{s,as\}])=\mathbf{p}_{y(s)}\mathbf{P}^{a}_{y(s),y(as)}=0$, so there is an open set containing $y$ with null measure, i.e., $y\not\in\supp(\nu_m)$. Thus $\supp(\nu_m)\subseteq \phi_m(\mathcal{A}^S)$.

    Notice that the Markov measure $\nu_m \in M_S((\left.X\right\vert_{B_m})^{S})$ is $\freegroup_d$-extensible (by Lemma \ref{lemma:Ext_Fd}) and trivially conjugate to a measure $\nu_m' \in M_S(\phi_m(\mathcal{A}^S))$. Then, since $\freegroup_d$-extensibility is preserved under factors (by Proposition \ref{prop:isomorphisms}), the measure $\nu'_m$ is $\freegroup_d$-extensible. Let $\mu_m=(\phi_m^{-1})_*\nu'_m$, i.e., $\mu_m(A)=\nu'_m(\phi_m(A))$ for all $A\in \mathcal{B}(\mathcal{A}^{S})$. Since $\phi^{-1}_m: \phi_m(\mathcal{A}^S) \to \mathcal{A}^S$ is a conjugacy and the measure $\nu'_m$ is $\freegroup_d$-extensible, the measure $\mu_m$ is $\freegroup_d$-extensible (again by Proposition \ref{prop:isomorphisms}).

    We need to prove that $\mu_m\to\mu$ in the weak-* topology. If $F\subseteq S$ is finite and $x\in \mathcal{A}^{F}$, there is some $m_0 \geq 0$ such that $F\subseteq B_{m}$ for every $m \geq m_0$, so 
    \begin{align*}
        \phi_m([x\,;F])=\phi_m\left(\bigsqcup_{y\in \mathcal{A}^{B_{m}-F}}[y\wedge x\,;B_{m}]\right)=\phi_m(\mathcal{A}^S) \cap\ \bigsqcup_{y\in \mathcal{A}^{B_{m}-F}}[y\wedge x\,;\varepsilon],
    \end{align*}
    whence
    \begin{align*}
        \mu_m([x\,;F])&=\overline{\nu}_m\left(\bigsqcup_{y\in \mathcal{A}^{B_{m}-F}}[y\wedge x\,;\varepsilon]\right)=\nu_m\left(\bigsqcup_{y\in \mathcal{A}^{B_{m}-F}: y\wedge x \in \left.X\right\vert_{B_m}}[y\wedge x\,;\varepsilon]\right)\\
        &=\sum_{y\in \mathcal{A}^{B_{m}-F}: y\wedge x \in \left.X\right\vert_{B_m}}\mathbf{p}_{y\wedge x}
        =\sum_{y\in \mathcal{A}^{B_{m}-F}}\mu([y\wedge x\,;B_{m}])=\mu([x\,;F]).
    \end{align*}
     As a result, $\mu_m([x;F])$ is eventually $\mu([x;F])$ as $m\to \infty$ for every cylinder, implying that $\mu_m\to\mu$ in the weak-* topology. Therefore, as $\Ext_{\freegroup_d}(\alphabet^{S},S)$ is weak-* closed, we conclude that $\mu\in \Ext_{\freegroup_d}(\alphabet^{S},S)$. Since $\alphabet$ and $\mu$ were arbitrary, by Proposition~\ref{prop: characterization of MEC} it follows that $S$ is MEC.
\end{proof}

We now combine the results of the left reversible and free case to prove Theorem~\ref{thm:C}.

\thirdtheorem*

\begin{proof}
Assume first that $S$ is left amenable and residually a finite group. Then, $S$ is embeddable and, in particular, bicancellative. Let $\freeSgroup=(\Gamma,\gamma)$ be the free $S$-group. By \cite[Lemma~1]{Donnelly0}, $S$ is left reversible. Therefore, by Theorem~\ref{thm:extensibility_reversible}, $S$ is MEC. By Proposition~\ref{prop:epa_property_amenable}, $\Gamma$ has the EPA property. We conclude by appealing to Theorem~\ref{thm:characterization of the PA property}. 

Now assume that $S = \F_d^+$, the free monoid on $d$ generators. By Proposition~\ref{prop:extensibility_free_case}, $S$ is MEC. By  \cite[Theorem~3.4]{bowen2003periodicity}, $\F_d$ has the PA property. We conclude by appealing again to Theorem~\ref{thm:characterization of the PA property}. 
\end{proof}

\begin{remark}
In \cite[Theorem 3.1]{kechris2012weakcontainment} it is proven that $\F_d$ has the so-called EMD property for every integer $d\geq 1$. It seems to be the case that the EMD property implies the EPA property. If this holds, the conclusion of part (2) of Theorem \ref{thm:C} strengthens, establishing the EPA property for $\F_d^+$ for every $d\geq 1$.
\end{remark}

\subsection{Extensibility to receiving groups in the free case} 

Let us first consider the following general obstruction to extensibility based on the support of an invariant measure.

\begin{proposition}
\label{prop:support_extension}
Let $S\acts X$ be a continuous action, let $\mu\in M_S(X)$, and let $\receivingSgroup$ be a receiving $S$-group. If $\bar{\mu}\in M_G(X_\receivingSgroup)$ is such that $\pi_*\bar{\mu}=\mu$, then $\supp(\bar{\mu})\subseteq (\supp(\mu))_\receivingSgroup$. In particular, if $(\supp(\mu))_\receivingSgroup=\varnothing$, then $S\acts (X,\mu)$ cannot be p.m.p. $\receivingSgroup$-extensible.
\end{proposition}

\begin{proof} 
    Note that if $\bar{x}=(x_h)_{h\in G}\in \supp(\bar{\mu})\subseteq X_{\receivingSgroup}$ and $U_h\subseteq X$ is an open neighborhood of $x_h$ for $h\in G$, then $h^{-1}\pi^{-1}U_h$ is an open neighborhood of $\bar{x}$, which means
    $$\mu(U_h)=\bar{\mu}(\pi^{-1}U_h)=\bar{\mu}(h^{-1}\pi^{-1}U_h)>0,$$
    so $x_h\in \supp(\mu)$ for all $h\in G$. Hence $\supp(\bar{\mu})\subseteq X_{\receivingSgroup}\cap (\supp(\mu))^G= (\supp(\mu))_{\receivingSgroup}$. Finally, if $(\supp(\mu))_\receivingSgroup=\varnothing$ and $\bar{\mu}\in M_G(X_\receivingSgroup)$ is such that $\pi_*\bar{\mu}=\mu$, then $(\supp(\mu))_{\receivingSgroup}=\varnothing$, so $\supp(\bar{\mu})=\varnothing$, which is absurd. This shows that there is no such $\bar{\mu}$, so by Remark~\ref{rem: Ext extensibility} we conclude that $S\acts (X,\mu)$ is not p.m.p. $\receivingSgroup$-extensible.
\end{proof}

In view of Proposition~\ref{prop:support_extension}, a continuous action $S\acts X$ with $X_\receivingSgroup = \varnothing$ that admits an $S$-invariant measure provides an example of a measure that is not $\receivingSgroup$-extensible. If $\freeSgroup=(\Gamma,\gamma)$ is the free $S$-group, $\Gamma$ is residually finite, and $\receivingSgroup\not\simeq \freeSgroup$, then there exist a finite alphabet $\alphabet$ and an $S$-periodic point $x\in \alphabet^S$ such that the continuous action $S\acts Sx$ satisfies $(Sx)_\receivingSgroup=\varnothing$, so we may take the periodic measure associated with $x$ (see Proposition~\ref{prop:extensibility_residually_finite}).

If $S$ is bicancellative and left reversible, there is only one receiving $S$-group up to isomorphism of $S$-groups, namely, the group of right fractions $\GrRightFrac{S}$, and every $S$-invariant measure is $\GrRightFrac{S}$-extensible. Therefore, we focus on the free case. Our last main result characterizes the free $S$-group in terms of the extensibility of $S$-invariant measures. In particular, it tells us that if $\receivingSgroup\not\simeq \freegroup_d$, there is a fully supported Markov measure on some full shift $\alphabet^S$---which is trivially topologically $\receivingSgroup$-extensible---that is not $\receivingSgroup$-extensible.

\fifththeorem*

\begin{proof}
Write $\receivingSgroup=(G,\eta)$. Assume (i). By Proposition~\ref{prop:extensibility_free_case}, $S=\F_d^+$ is MEC. From Proposition~\ref{prop:isomorphisms} we know that every p.m.p. action $S\acts (X,\mu)$ is $\receivingSgroup$-extensible. It then follows from Proposition~\ref{prop:reduction_to_the_natural_extension} and the identification between $(\alphabet^S)_{\receivingSgroup}$ and $\alphabet^G$, that $\Ext_{\receivingSgroup}(\alphabet^S,S)=M_S(\alphabet^S)$ for every finite alphabet $\alphabet$. Thus (i) implies (ii). 

That (ii) implies (iii) is direct. To see that (iii) implies (i), assume that (i) does not hold and let us show that (iii) does not hold either. Indeed, since $\receivingSgroup\not\simeq \freegroup_d$, there is a surjective morphism $\theta\colon \F_d\to G$ with non-trivial kernel and such that $\theta\circ \iota=\eta$, so we may choose an element $w\in\text{ker}(\theta)-\{ \varepsilon\}$. Consider a $2\times 2$ matrix representation of $\F_d$ generated by $d$ matrices $A_{1},\dots,A_{d}$. Write $w=A_{i_1}^{\epsilon_1}\cdots A_{i_n}^{\epsilon_n}$ for $n \geq 1$ and $\epsilon_j \in \{-1,1\}$, $j = 1,\dots,n$, and choose a sufficiently large positive prime $p$ such that $A_{i_1}^{\epsilon_1}\cdots A_{i_n}^{\epsilon_n}\neq I_2 \bmod p$. Let $\alphabet=\Z_p^2$ and $0<\delta<1/(p^2-1)$. Define, for $A \in \{A_1,\dots,A_d\}$ and $\mathbf{u},\mathbf{v}\in \alphabet$,
$$\mathbf{P}^{A}_{\mathbf{u},\mathbf{v}}=
\begin{cases}
    1-(p^2-1)\delta & \text{if }\mathbf{v}=A\mathbf{u} \bmod p,\\
    \delta & \text{otherwise}.
\end{cases}$$
and let $\mathbf{p}\in [0,1]^{p^2}$ be the uniform distribution upon $\alphabet$, i.e., $\mathbf{p}_{\mathbf{u}}=1/p^2$ for all $\mathbf{u}\in \alphabet$. Since $A$ is invertible and $p$ is prime, $A$ is invertible modulo $p$ as well. Then, for a fixed $\mathbf{v}\in \alphabet$, there is a unique $\mathbf{u}\in \alphabet$ such that $\mathbf{v}=A\mathbf{u} \bmod p$, so
$$(\mathbf{p}\mathbf{P}^{A})_{\mathbf{v}}=\sum_{\mathbf{u}\in \alphabet}\mathbf{p}_{\mathbf{u}}\mathbf{P}^{A}_{\mathbf{u},\mathbf{v}}=\frac{(p^2-1)\delta}{p^2}+\frac{1-(p^2-1)\delta}{p^2}=\frac{1}{p^2},$$
yielding $\mathbf{p}\mathbf{P}^{A}=\mathbf{p}$. Thus, letting $\mathbf{P}=(\mathbf{P}^{A_i})_{i=1}^d$, the $(\mathbf{p},\mathbf{P})$-Markov measure $\mu$ on $\alphabet^{S}$ is $S$-invariant by Lemma~\ref{prop:georgii}. Since $\mathbf{p}>0$ and $\mathbf{P}^{A}>0$ for every $A\in\{A_1,\dots,A_d\}$, $\mu$ is fully supported.

Assume that there is a $\receivingSgroup$-extension $\bar{\mu}\in M_G(\alphabet^G)$ of $\mu$. If $\overline{x}\in \alphabet^G$, $g\in G$, and $A \in \{A_1,\dots,A_d\}$, then 
\begin{align*}
\bar{\mu}\big{(}[\overline{x}\,;\{g,\eta(A)g\}]\big{)}  &   =\bar{\mu}\big{(}g^{-1}[g\cdot\overline{x}\,;\{1_G,\eta(A)\}]\big{)}    \\
&   =\bar{\mu}\big{(}[g\cdot\overline{x}\,;\{1_G,\eta(A)\}]\big{)}    \\
    &   =\mu\left(\left[\pi(g\cdot\overline{x})\,;\left\{\varepsilon,A\right\}\right]\right) \\
    &   =\mathbf{p}_{\hspace{1pt}\overline{x}(g)}\mathbf{P}^{A}_{\overline{x}(g),\overline{x}(\eta(A)g)},
\end{align*}
    as $\pi(g\cdot \overline{x})(t)=\overline{x}(tg)$ for all $t\in S$. Analogously, we have that
    $$\bar{\mu}\big{(}[\overline{x}\,;\{g,\eta(A)^{-1}g\}]\big{)}=\mathbf{p}_{\hspace{1pt}\overline{x}(\eta(A)^{-1}g)}\mathbf{P}^{A}_{\overline{x}(\eta(A)^{-1}g),\overline{x}(g)}.$$
    
    Since $A_{i_1}^{\epsilon_1}\cdots A_{i_n}^{\epsilon_n}\neq I_2 \bmod p$, there exists $\mathbf{v^*}\in \alphabet$ such that $A_{i_1}^{\epsilon_1}\cdots A_{i_n}^{\epsilon_n}\mathbf{v^*}\neq \mathbf{v^*} \bmod p$. Consider the finite subset $C=\{\eta(A_{i_k})^{\epsilon_k}\cdots \eta(A_{i_n})^{\epsilon_n}: 1\leq k\leq n\}$ of $G$. Since $\theta\circ \iota=\eta$, we have that $\eta(A_{i_1})^{\epsilon_1}\cdots \eta(A_{i_n})^{\epsilon_n}=\theta(w)=1_G$, so $C$ can be viewed as a cycle in $G$. Let $\tau\in \alphabet^C$ be a configuration such that $\tau(1_G)=\mathbf{v}^*$, and assume that  $\tau(\eta(A_{i_k})^{\epsilon_k}\cdots \eta(A_{i_n})^{\epsilon_n})=A_{i_k}^{\epsilon_k}\tau(\eta(A_{i_{k+1}})^{\epsilon_{k+1}}\cdots \eta(A_{i_n})^{\epsilon_n})$ for every $0\leq k< n$, where we consider $A_{i_0}=A_{i_n}$ and $\epsilon_0=\epsilon_n$. Iteratively, we see that 
    $$\mathbf{v^*}=\tau(1_G)=\tau(\eta(A_{i_1})^{\epsilon_1}\cdots \eta(A_{i_n})^{\epsilon_n})=A_{i_1}^{\epsilon_1}\cdots A_{i_n}^{\epsilon_n}\mathbf{v}^*\neq \mathbf{v}^*,$$
    a contradiction, so it must be the case that there exist $g\in C$ and $0\leq k<n$ such that $\eta(A_{i_k})^{\epsilon_k}g\in C$ and $\tau(\eta(A_{i_k})^{\epsilon_k}g)\neq A_{i_k}^{\epsilon_k}\tau(g)$. Hence,
$$
\bar{\mu}([\tau;C]) \leq \bar{\mu}\big{(}[\tau\,;\{g,\eta(A_{i_k})^{\epsilon_k}g\}]\big{)}=\mathbf{p}_{\tau(g)}\mathbf{P}^{A_{i_k}}_{\tau(g),\tau(\eta(A_{i_k})g)}
= \frac{\delta}{p^2}
$$
if $\epsilon_k=1$, and similarly $\bar{\mu}([\tau;C])\leq \delta/p^2$ if $\epsilon_k=-1$ as well. Therefore,
$$\frac{1}{p^2}=\mu([\mathbf{v^*}\,;\varepsilon])=\bar{\mu}([\mathbf{v^*}\,;1_G])
        = \sum_{\tau\in \alphabet^{C}:\tau(1_G)=\mathbf{v^*}}\bar{\mu}([\tau;C])\leq (p^2)^{n-1}\frac{\delta}{p^2}=p^{2n-4}\delta,$$
        which leads to a contradiction if we take $\delta < \frac{1}{p^{2n-2}}$.
\end{proof}

 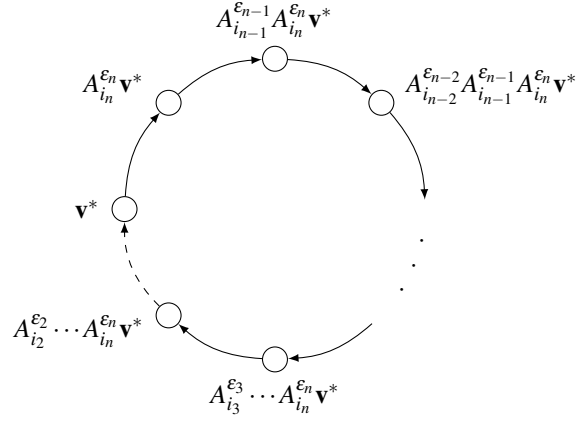
\begin{figure}[h]
\centering

		\begin{tikzpicture}
			\node[shape=circle,draw] (N1) at (180:2cm) {};
			\node at (180:2.5cm) {$\mathbf{v^*}$};
			\node[shape=circle,draw] (N2) at (135:2cm){};
			\node[left] at (135:2.3cm) {$A_{i_n}^{\epsilon_n}\mathbf{v^*}$};
			\node[shape=circle,draw] (N3) at (90:2cm) {};
			\node at (90:2.5cm) {$A_{i_{n-1}}^{\epsilon_{n-1}}A_{i_n}^{\epsilon_n}\mathbf{v^*}$};
			\node[shape=circle,draw] (N4) at (45:2cm) {};
			\node[right] at (45:2.3cm) {$A_{i_{n-2}}^{\epsilon_{n-2}}A_{i_{n-1}}^{\epsilon_{n-1}}A_{i_n}^{\epsilon_n}\mathbf{v^*}$};
			\node (V1) at (0:2cm) {};
			\node (V2) at (-45:2cm) {};
			\node at (-12.5:2cm) {$\cdot$};
			\node at (-22.5:2cm) {$\cdot$};
			\node at (-32.5:2cm) {$\cdot$};
			\node[shape=circle,draw] (T1) at (-90:2cm) {};
			\node at (-90:2.5cm) {$A_{i_{3}}^{\epsilon_{3}}\cdots A_{i_n}^{\epsilon_n}\mathbf{v^*}$};
			\node[shape=circle,draw] (T2) at (-135:2cm) {};
			\node[left] at (-135:2.3cm) {$A_{i_{2}}^{\epsilon_{2}}\cdots A_{i_n}^{\epsilon_n}\mathbf{v^*}$};
			\draw[-latex, bend left=19] (N1) to (N2);
			\draw[-latex, bend left=19] (N2) to (N3);
			\draw[-latex, bend left=19] (N3) to (N4);
			\draw[-latex, bend left=19] (N4) to (V1);
			\draw[-latex, bend left=19] (V2) to (T1);
			\draw[-latex, bend left=19] (T1) to (T2);
			\draw[-latex,dashed, bend left=22.5] (T2) to (N1);
		\end{tikzpicture}

\caption{In this cycle in $G$, transitions corresponding to solid edges all have high probability. As $A_{i_1}^{\epsilon_1}\cdots A_{i_n}^{\epsilon_n}\mathbf{v^*}\ne\mathbf{v^*}$, the final transition is thus forced to be of very low probability.}
\label{figure:Baumslag_cycle}
\end{figure}

\subsection*{Acknowledgments}

We would like to thank Sohail Farhangi for several comments and helpful observations, and for pointing out the results on natural extensions appearing in \cite{farhangi1} and \cite{rodriguez2025}.

\subsection*{Funding}

R.B. was partially supported by ANID/FONDECYT Regular 1240508 and Avanza UC AV25085. A.B.-G. was supported by ANID/FONDECYT Postdoctorado 3230159 (year 2023). M.D.-E. was partially supported by Beca de Magíster Nacional ANID 22230917 (year 2023), and by the Deutsche Forschungsgemeinschaft (DFG, German Research Foundation) under Germany's Excellence Strategy EXC 2044/2 –390685587, Mathematics Münster: Dynamics–Geometry–Structure. 

\bibliographystyle{abbrv}
\bibliography{references}

\end{document}